\documentclass[a4paper,fleqn]{cas-sc}
\usepackage[T1]{fontenc}     % 
\usepackage{fix-cm}

\DeclareSymbolFont{CMletters}{OML}{cmm}{m}{it}
\SetSymbolFont{CMletters}{bold}{OML}{cmm}{b}{it}

\DeclareMathSymbol{\ell}{\mathord}{CMletters}{"60}

\usepackage{bm}

\usepackage{graphicx}
\usepackage{enumitem,xcolor}
\usepackage{amsfonts,amssymb}
\usepackage[ruled,vlined,linesnumbered,norelsize]{algorithm2e}
\hypersetup{hidelinks,
breaklinks=true,
colorlinks=true,
linkcolor=blue,
citecolor=blue,
urlcolor=blue,
bookmarksopen=false,
pdftitle={Title},
pdfauthor={Author}}

\usepackage{subcaption}
\usepackage{placeins}

\DeclareMathOperator{\conv}{conv}

\DeclareMathOperator*{\argmax}{argmax}

\newcommand{\Po}{\mathcal{P}}
\newcommand{\D}{\mathcal{D}}

\newtheorem{theorem}{Theorem}%[section]
\newtheorem{remark}[theorem]{Remark}

\newtheorem{proposition}[theorem]{Proposition}
\newtheorem{corollary}[theorem]{Corollary}
\newtheorem{example}[theorem]{Example}
\newtheorem{observation}[theorem]{Observation}
\newproof{proof}{Proof}
 
\makeatletter
\renewcommand*{\top}{%
  {\mathpalette\@transpose{}}%
}
\newcommand*{\@transpose}[2]{%
  \raisebox{\depth}{\scriptsize $\m@th#1\mathsf{T}$}%
}
\newcommand{\B}[2][3.39em]{\makebox[#1][c]{\ensuremath{#2}}}
\makeatother

\allowdisplaybreaks[1]

\begin{document}

\let\WriteBookmarks\relax
\def\floatpagepagefraction{1}
\def\textpagefraction{.001}
\shorttitle{On disjunction convex hulls for generalized cross polytopes}
\shortauthors{Y.~Qu, J.~Lee}

\title[mode = title]
{On disjunction convex hulls for %\texorpdfstring{\\}{} 
generalized cross polytopes} 
\tnotemark[1]
\tnotetext[1]{A short preliminary version of this work is to appear 
in the proceedings of ISCO 2026,  Springer 
\emph{Lecture Notes in Computer Science.}}

\author[1]{Yushan Qu}[orcid=0000-0002-2862-4224]
\ead{yushanqu@umich.edu}
\ead[URL]{https://yushanqu.github.io}

\author[1]{Jon Lee}[orcid=0000-0002-8190-1091]
\ead{jonxlee@umich.edu}
\ead[URL]{https://sites.google.com/site/jonleewebpage/}

\affiliation[1]{organization={University of Michigan},city={Ann Arbor},state={Michigan},country={U.S.A.},}

\begin{abstract}
We continue the study of the natural polytope $\D$ in $\mathbb{R}^{n+d}$ associated with the
disjunction of a set of $n+1$ polytopes in $\mathbb{R}^d$,
managed by $n$ binary variables. Already $\D$ had been characterized for 
arbitrary $n\geq 1$ and 
(i) $d\in\{1,2\}$, and (ii) for a broad generalization of hyper-rectangles.
In both cases, the complete characterization
employs full optimal big-M lifting.
Here, we give 
a complete description of $\D$ for the case of $n\!=\!1$ and arbitrary
$d$, when the (two) polytopes are arbitrary generalized cross polytopes. 
Furthermore, we characterize when 
our complete description 
employs only optimal big-M lifting.
For $n>1$, we generalize the family of facet-describing inequalities used for $n=1$. Finally, we 
carry out some computational 
experiments demonstrating the 
value of our theoretical results. 
\end{abstract}


\maketitle   

%%%%%%%%%%%%%%%

\section{Introduction}

For families of $n+1$ input polytopes in $\mathbb{R}^d$, we are interested in giving descriptions, by linear inequalities, of the convex hull $\D$ of
the natural polytope in $\mathbb{R}^{n+d}$, associated to the disjunction of the input polytopes, as managed by $n$ binary indicator variables --- we will make this precise in \S\ref{sec:def}.
In \cite{QL_DAM2025}, we gave a complete description for $n\geq 1$ and  $d\in\{1,2\}$, via
full optimal big-M lifting inequalities (starting from facet-describing inequalities of the input polytopes). Additionally, we demonstrated that even for the case of two simplices in $\mathbb{R}^3$, the polytope $\D$ can have facets that are not described by lifting. The other main result of \cite{QL_DAM2025}
is that for $n+1$ input polytopes having ``the same shape'', $\D$ has a very compact description 
via full optimal big-M lifting inequalities; see \S\ref{sec:prior_characterizations} for a precise statement.
This latter result includes the case in which all of the input polytopes are (axis-aligned)
hyper-rectangles, first established in \cite{QL_ISCO2024}. With this result in mind, we are motivated to consider other natural families of input polytopes, and to this end, we introduce ``generalized cross polytopes''
(defined in \S\ref{sec:newresults_cross}) --- a polar opposite to the case of hyper-rectangles.
While a hyper-rectangle in $\mathbb{R}^d$ has $2^d$ extreme points and $2d$ facets,
a ``generalized cross polytope''  in $\mathbb{R}^d$ has $2d$ extreme points and $2^d$ facets.
We are able to obtain a complete description of $\D$ for the case of $n=1$ (two input polytopes) and general $d\geq 1$; further, in this setting, we provide a characterization of when optimal big-M lifting suffices to describe $\D$.

\medskip
\noindent {\bf Organization and Contributions.}
In \S\ref{sec:def}, 
we give the main definitions and set the general notation for our setting.
In \S\ref{sec:priorres}, we give a very brief survey of previous results to
position our new results in the context of the literature.
In \S\ref{sec:newresults_cross}, we define ``generalized cross polytopes''
and present our main results.
We give 
a complete description of $\D$ for the case of $n\!=\!1$ and arbitrary
$d$, when the (two) polytopes are arbitrary generalized cross polytopes. 
We characterize when 
our complete description 
employs only optimal big-M lifting inequalities. We note that even in the cases where
only optimal big-M lifting inequalities are employed, our complete description result
is not implied by prior results. For $n>1$, we give a broad generalization of the family of facet-describing inequalities used for $n=1$. 
In \S\ref{sec:exper}, we present results of computational experiments demonstrating the 
value of our theoretical results. 
In \S\ref{sec:outlook}, we describe our next steps.
Appendices \ref{app:facet}, \ref{app:lift}, and \ref{app:ex}
contain some deferred details.

\medskip
\noindent
{\bf Standard terminology and notation concerning polytopes.} 
We assume familiarity with some aspects of the theory of polytopes and of integer programming; see \cite{ziegler} and \cite{cornrio,CCZ_Book,NWbook,balas2018disjunctive}, respectively.
We say that $\alpha^\top x \leq \beta$ is 
\emph{valid} for a polytope $\Po\subset \mathbb{R}^r$ if it is satisfied by
all points of $\Po$. 
For a valid inequality $\alpha^\top x \leq \beta$  of $\Po$, the \emph{face described} is $\Po\cap \{ x \in \mathbb{R}^r ~:~ \alpha^\top x = \beta\}$.
The \emph{dimension} of $\Po$ is one less than the maximum number of affinely-independent points in $\Po$, and in particular, it is
\emph{full dimensional} if its dimension is 
the same as the ambient dimension $r$. 
A \emph{facet} of $\Po$ is a face having dimension one less than that of $\Po$. 
A full-dimensional polytope has an essentially
unique minimal description (up to positive scaling) using 
facet-describing inequalities. 

%%%%%%%%%%%%%%%%%%%%

\section{Definitions and an elementary result}\label{sec:def}

Our setup is the same as in \cite{QL_ISCO2024,QL_DAM2025}, and we refer the
reader to \cite{QL_DAM2025} for many links to the literature.
For the positive integer $d$, we let $[d]:=\{1,2,\ldots,d\}$. 
For the positive integer $n$, we let $N_0 := \{0,1,...,n\}$ and  $N:= \{1,...,n\}$. 
 For $d\geq 1$ and $n\geq 1$, we consider $n+1$ full-dimensional polytopes $\Po_i:=\{ x\in\mathbb{R}^d ~:~ \bm{A}^i x \leq \bm{b}^i\}$, for  $i\in N_0$\,. 
The number of columns of each matrix $\bm{A}^i$ is $d$, and the number of rows of each $\bm{A}^i$ agrees with the number of elements of the corresponding vector $\bm{b}^i$\,. We  do not assume that the $\Po_i$ are pairwise disjoint,
but it may be the case.
For convenience, we assume that 
 each inequality defining each $\Po_i$ describes a unique facet of $\Po_i$\,. 

 We employ binary variables $z_i$\,, for 
 $i\in N$.
The interpretation of $z_i=1$ is that $x\in\mathbb{R}^d$ is required to be in $\Po_i$\,, Further, if $z_i=0$ for all $i \in N$, then  $x$ is required to be in $\Po_0$\,. 
We then require that 
$x$ is enforced to be 
in (at least) one of the $\Po_i$ using
$
\textstyle \sum_{i\in N} z_i \leq 1.\label{eq:choice}
$
Because the $\Po_i$ may overlap, it can well be that $z_i=0$, but $x\in \Po_i$ for some $i\in N$. Likewise, we can have $z_i=1$ for 
 some choice of $i\in N$, but $x\in \Po_0$ or other $\Po_j$\,.

Throughout, whenever we write $\genfrac(){0pt}{1}{x}{z}$, 
it can be assumed that 
$x\in\mathbb{R}^d$ and $z\in\mathbb{R}^n$.
For $i\in N$,  let $\mathbf{e}_i$ denote the $i$-th standard unit vector in $\mathbb{R}^n$,
and additionally for convenience, let  $\mathbf{e}_0$
denote the zero vector in $\mathbb{R}^n$.
For $i\in N_0$\,, let $\bar{\Po}_i:=\left\{ \genfrac(){0pt}{1}{\hat x}{\mathbf{e}_i}~:~ \hat x\in \Po_i\right\}$,
which is the polytope $\Po_i\subset \mathbb{R}^d$ lifted
to $\mathbb{R}^{d+n}$ by setting $z=\mathbf{e}_i$\,.
For $i\in N_0$\,, let $\mathcal{X}_i$ be the (finite) set of extreme points of $\Po_i$\,, 
and let $\bar{\mathcal{X}}_i:=\left\{\genfrac(){0pt}{1}{\hat x}{\mathbf{e}_i} ~:~ \hat x\in \mathcal{X}_i\right\}$, the (finite) set of extreme points of 
$\bar{\Po}_i$\,. 
Finally, let
$
\D:= \conv
\left\{
\bar{\Po}_i ~:~ i\in N_0
\right\}
= \conv
\left\{
\bar{\mathcal{X}}_i ~:~ i\in N_0
\right\}.
$
For various families of 
polytopes, we are interested in characterizing $\D$ by linear inequalities.

Before continuing, we state an elementary result
concerning an affine bijection related to translations of the $n+1$ input polytopes. This result allows us to translate the input polytopes to any convenient locations, and then polyhedral
results for the translated polytopes
map back to the original polytopes according to the inverse of the affine bijection. 

\begin{proposition}\label{prop:translations}
    Let $t^i \in \mathbb{R}^d$ be arbitrary,
    and let $\Po_i' := \Po_i + t^i$\, for
    $i\in N_0$\,. Let $\D' := {\rm conv}\left( \bar{\Po}'_0\,, \cdots, \bar{\Po}'_n \right)$. We consider the affine bijection $f: \mathbb{R}^{d+n}\rightarrow \mathbb{R}^{d+n}$ defined by $f\genfrac(){0pt}{1}{x}{z} := \genfrac(){0pt}{1}{x + \sum_{j\in N} z_j t^j + \left(1-\sum_{j\in N} z_j\right)t^0}{ z}$. Then, $\D' = f(\D)$.
\end{proposition}

\begin{proof}
For $i\in N_0$ and $\genfrac(){0pt}{1}{\hat x}{\mathbf{e}_i} \in \bar{\mathcal{X}}_i$\,, it is easy to check that 
$f\genfrac(){0pt}{1}{\hat x}{\mathbf{e}_i}
$ is the extreme point 
$\genfrac(){0pt}{1}{\hat x+t^i}{\mathbf{e}_i}$ of $\D'$. Moreover, we can verify that $f$ is invertible:
$f^{-1}
\genfrac(){0pt}{1}{x}{z} =
\genfrac(){0pt}{1}{ x - \sum_{j\in N} z_j t^j - (1-\sum_{j\in N} z_j)t^0}{z} 
$. It is easy to conclude now that $\D' = f(\D)$.
\qed    
\end{proof}

%%%%%%%%%%%%

\section{Prior Results}\label{sec:priorres}

\subsection{General results}

We have the following positive complexity result, when the 
dimension $d$ is considered fixed.

\begin{theorem}[\protect{\cite[Theorem 11]{QL_DAM2025}}]\label{thm:computing}
    For $n+1$ polytopes $\Po_i \subset\mathbb{R}^d$, for $i \in N_0$\,,
    given by inequality descriptions, we can compute 
    a facet-describing inequality system for $\D$ in time polynomial
    in the input size, considering $d$ to be fixed.
\end{theorem}

The following results seeks to understand the facet structure of 
$\D$. 

\begin{theorem}[\protect{\cite[Theorem 2]{QL_DAM2025}}]
\label{thm:fulldim}
Suppose that $d\geq 1$. 
If all $\Po_i\subset\mathbb{R}^d$ are nonempty, for $i\in N_0$\,, and at least one is full dimensional, then 
$\D$ is full dimensional.
\end{theorem}

\begin{theorem}[\protect{\cite[Theorem 4]{QL_DAM2025}}]
\label{thm:nonneg}
Suppose that $d\geq 1$. 
If all $\Po_i\subset\mathbb{R}^d$ are nonempty, for $i\in N_0$\,, and $\Po_0$ is full dimensional, then for all $j\in N$, $z_j\geq 0$ is facet describing for $\D$. 
\end{theorem}

\begin{theorem}[\protect{\cite[Theorem 5]{QL_DAM2025}}]
\label{thm:sumz}
Suppose that $d\geq 1$. 
If all $\Po_i\subset\mathbb{R}^d$ are nonempty, for $i\in N_0$\,,
and $\Po_k$ is full dimensional for some $k\in N$, then $\sum_{j\in N} z_j \leq 1$
is facet describing for $\D$. 
\end{theorem}
As in \cite{QL_DAM2025},
we refer to facets of $\D$ other than 
those described by Theorems \ref{thm:nonneg}
and \ref{thm:sumz} as \emph{vertical facets}.

The ``big-M'' method is a well known and
classical method for treating disjunctions 
by using binary variables at the modeling level (see \cite[Section 26-3.I, parts (b--d,g)]{Dantzig}, \cite[Sections 8.1.2--3]{LeeLP}, for example). ``Lifting'' is a general well-known technique
for extending linear inequalities in disjunctive settings.
See, for example, \cite[Chapter II.2, Section 1]{NWbook}. In our setting, there is a close connection between big-M modeling and lifting.
For $j\in N_0$\,, let
$
M_j := \min \{ \beta -\alpha^\top x ~:~ x \in \Po_j\}
$.
Starting with a valid inequality 
$
\alpha^\top x \leq \beta $
for $\Po_0$\,,
we can obtain the \emph{full optimal big-M lifting inequality}
$\alpha^\top x + \sum_{i\in N} M_i z_i \leq \beta$, which is valid for $\D$.
Starting instead with 
 a valid inequality 
$
\alpha^\top x \leq \beta $
for $\Po_k$\,, for some $k\in N$,
we can obtain the  \emph{full optimal big-M lifting inequality}
$\alpha^\top x + \sum_{i\in N \setminus\{k\}} (M_i-M_0) z_i -M_0 z_k\leq \beta  - M_0$\,, which is valid for $\D$.
Some brief points on terminology:
(i) ``big-M'' relates to the fact that
the  coefficients $M_j$ are chosen independently of one another; 
(ii) ``optimal'' refers to the fact that 
the ``big-M'' coefficients $M_j$ are chosen to be as small as possible;
(iii) ``full'' refers to lifting the coefficients on \emph{all} of the $z_i$\, so when $n=1$, we omit the word ``full''.

\begin{theorem}[\protect{\cite[Theorem 3]{QL_DAM2025}}]
\label{thm:facet_lift}
Suppose that $d\geq 1$. 
If all $\Po_i\subset\mathbb{R}^d$ are nonempty, for $i\in N_0$\,,  then 
the full optimal big-M lifting of a facet-describing inequality obtained from a full-dimensional $\Po_k$ (for any $k\in N_0$)
is facet describing for $\D$. 
\end{theorem}

%%%%%%%%%%%%%%%%%%%%%%%%%

\subsection{Low dimension and common shape via lifting}\label{sec:prior_characterizations}

The most significant contributions of \cite{QL_DAM2025}
is to identify some broad classes of polytopes
for which full optimal big-M lifting is enough to characterize $\D$ via linear inequalities.
The first result considers low dimensions. 

\begin{theorem}[\protect{\cite[Theorem 6]{QL_DAM2025}}]\label{thm:hull_low}
Suppose that  $n\geq 1$, $d\in\{1,2\}$,
all $\Po_i$ are nonempty, for $i\in N_0$\,, and at least one is full dimensional.
Then
the full optimal big-M lifting of
each facet-describing inequality defining  each $\Po_i$\,,
for
$i \in N_0$\,, together with the inequalities
$\sum_{i\in N} z_i \leq 1$ and  $z_i \geq 0$, for
$i \in N$,
gives the convex hull $\D$.
\end{theorem}   
We note that Theorem \ref{thm:hull_low}
does not extend to $d=3$; see \cite[Proposition 9]{QL_DAM2025},
as well as Example \ref{ex:2cross} in what follows.

Another broad situation where 
full optimal big-M lifting is enough to characterize $\D$ via linear inequalities
is, roughly speaking, when the $\Po_i$ (in arbitrary dimension $d$) all have the ``same shape'', in which case we
obtain a rather compact inequality description of $\D$:

\begin{theorem}[\protect{\cite[Theorem 15]{QL_DAM2025}}]\label{thm:hull_shape}
For $n \geq 1$ and 
$d \geq 1$, 
let $\Po_i:=\{ x\in\mathbb{R}^d ~:~ \bm{A} x \leq \bm{b}^i\}$
be full dimensional, where 
$\bm{b}^i\in\mathbb{R}^m$, for  $i\in N_0$\,.
We assume that the (single) matrix $\bm{A}$ has full column rank, and that for every row $\bm{A}_{j\cdot}$ of $\bm{A}$,
the inequality $\bm{A}_{j\cdot} x \leq \bm{b}^i_k$ describes a 
nonempty face of every $\Po_i$ and a 
facet of some $\Po_i$\,.
We further suppose:
\vskip-15pt

\begin{equation}
\tag{$\mathrm{\Phi}$}\label{Phi}
\begin{array}{l}
\mbox{For every
basic partition $\tau,\eta$ of the \emph{row} indices of the matrix $\bm{A}$,}\\
\mbox{we have that if }\bm{A}_{\eta\cdot} \bm{A}_{\tau \cdot}^{-1} \bm{b}^i_\tau \leq \bm{b}^i_\eta\,,
\mbox{ holds for some $i$ in $N_0$\,, then it}\\
\mbox{holds for every $i$ in $N_0$\,.}
\end{array}
\end{equation}
\vskip-7pt

\noindent Then
the full optimal big-M lifting inequalities 
$\bm{A} x + \sum_{i \in N} (\bm{b}^0 - \bm{b}^i) z_i \leq \bm{b}^0$\,, 
together with $\sum_{j\in N} z_j \leq 1$, and 
  $z_j\geq 0, \mbox{ for } j\in N$,
gives a minimal inequality description of the convex hull $\D$.
\end{theorem} 
Note that in the setting addressed by Theorem 
\ref{thm:hull_shape}, there is exactly one 
 full optimal big-M lifting inequality for
 each row of $\bm{A}$, independent of $n$. 

A simple special case of Theorem 
\ref{thm:hull_shape} is for hyper-rectangles:

\begin{corollary}[\protect{\cite[Theorem 7]{QL_ISCO2024},
\cite[Corollary 18]{QL_DAM2025}}]
\label{cor:hyper}
For $n \geq 1$ and 
$d \geq 1$, we consider  $n+1$ hyper-rectangles 
$\Po_j:= [\bm{\ell}_{j1}\,,\bm{u}_{j1}] \times \cdots \times [\bm{\ell}_{jd}\,,\bm{u}_{jd}]$,
for $j\in N_0$\,.
The full optimal big-M lifting inequalities 
$x_i + \textstyle  \sum_{j \in N} (\bm{\ell}_{0i} - \bm{\ell}_{ji}) z_j \geq \bm{\ell}_{0i}$\,, 
$\textstyle x_i + \sum_{j \in N} (\bm{u}_{0i} - \bm{u}_{ji}) z_j \leq \bm{u}_{0i}\,, \text{ for all } i = 1, \cdots, d$,
together with $\sum_{j\in N} z_j \leq 1$, and 
  $z_j\geq 0, \mbox{ for } j\in N$,
gives the convex hull $\D$.  
\end{corollary}
Of course in the case that $\bm{\ell}_{j\cdot}=-\mathbf{e}\in\mathbb{R}^d$ and 
$\bm{u}_{j\cdot}=\mathbf{e}\in\mathbb{R}^d$,
the polytope $\Po_j$ is the 
unit ball of the $\infty$-norm.

%%%%%%%%%%%%%

\section{Generalized cross polytopes}\label{sec:newresults_cross}

We are interested in exploring 
facet characterizations of $\D$ for scenarios
not covered by those in Section \ref{sec:prior_characterizations}.
As Corollary \ref{cor:hyper} covers
the unit ball of the $\infty$-norm, 
it is intriguing to consider a kind of polar opposite;
i.e., 
families of polytopes that include 
the unit ball of the $1$-norm, namely the standard cross polytope.

We define $n+1$ \emph{generalized cross polytopes}: For 
$d \geq 1$ and $\bm{a}^k\in \mathbb{R}^d_{++}$\,, for 
$k\in N_0$\,, we consider the generalized cross polytopes 
$\Po_k:= \{x \in \mathbb{R}^d: \sum_{j \in [d]} s_{j} \bm{a}^k_{j} x_j \leq 1, \mbox{ for all } 
s:=(s_1,\cdots,s_d) \in \{\pm1 \}^d\}$.
While each $\Po_k$ is defined via
$2^d$  linear inequalities,
we can also see it as
the solution set of the 
single nonlinear inequality
$ \sum_{j \in [d]} \bm{a}^k_{j} |x_j| \leq 1$.
The extreme points of $\Po_k$ are 
the $2d$ points 
\[
\hat{x}^{k,i} := \left\{\left( \genfrac{}{}{0pt}{1}{\pm \frac{1}{\bm{a}_i^k} \mathbf{e}_i}{\mathbf{e}_k} \right) \right\}, i \in [d] \mbox{ for }k\in N_0\,.
\]
Because our generalized cross polytopes are full dimensional, 
Theorems \ref{thm:computing}--\ref{thm:facet_lift}
apply. 

A (standard) \emph{cross polytope} (see \cite[p. 8]{ziegler}, for example),
which is the unit ball of the 1-norm,
has $\bm{a}^k:=\mathbf{e}$ and can be seen
as the solution set of
$ \sum_{j \in [d]} |x_j| \leq 1$.
We note that due to Proposition \ref{prop:translations},
our results in this section apply (under an appropriate affine
transformation) to \emph{translated} generalized cross polytopes as well.

\subsection{\texorpdfstring{$n=1$}{n=1}}\label{sec:n=1}

To work toward describing $\D$ by linear inequalities, we must set some notation.
For $j \in [d]$, we let $r_j :=
\left. \bm{a}^0_j \middle /\bm{a}^1_j\right.$\,.
In what follows, we assume, without loss of generality, that 
$r_1\leq r_2 \leq \cdots \leq r_d$\,.

We partition $[d]$, depending on the $r_j$\,, into blocks
 $\mathcal{L}(j) := \{i \in [d]: r_i < r_j\}$, 
 $\mathcal{E}(j) := \{i \in [d]: r_i = r_j\}$, 
 $\mathcal{G}(j) := \{i \in [d]: r_i > r_j\}$. 
For each $i\in[d]$, we define 
\[
\alpha_i(j):= 
\begin{cases}
\bm{a}_i^0= r_j \bm{a}_i^1\,, & \mbox{for } i \in \mathcal{E}(j); \\ 
\bm{a}_i^0\,, & \mbox{for } i \in \mathcal{L}(j); \\ 
r_j \bm{a}_i^1\,, & \mbox{for } i \in \mathcal{G}(j).
\end{cases}
\]

Define  $M_0(j) := \max_{ i \in [d]}\alpha_i(j)/\bm{a}^0_i $\,, $M_1(j) := \max_{ i \in [d]}\alpha_i(j)/\bm{a}^1_i $\,.
We first observe that \begin{align*}
\alpha_i(j)/\bm{a}^0_i :=& 
\begin{cases}
\frac{ \bm{a}_i^0}{\bm{a}^0_i} = \frac{r_j \bm{a}_i^1}{{\bm{a}^0_i}}\,, & \mbox{for } i \in \mathcal{E}(j); \\ 
\frac{ \bm{a}_i^0}{\bm{a}^0_i}\,, & \mbox{for } i \in \mathcal{L}(j); \\ 
\frac{r_j \bm{a}_i^1}{\bm{a}^0_i}\,, & \mbox{for } i \in \mathcal{G}(j).
\end{cases} \\[5pt]
=& 
\begin{cases}
1 = \frac{r_j}{r_i}\,, & \mbox{for } i \in \mathcal{E}(j); \\ 
1\,, & \mbox{for } i \in \mathcal{L}(j); \\ 
\frac{r_j}{r_i}\,, & \mbox{for } i \in \mathcal{G}(j).
\end{cases}
\end{align*}
If $i \in \mathcal{L}(j) \cup \mathcal{E}(j)$, then $r_i \leq r_j$, so $\frac{r_j}{r_i} \geq 1$, and $\alpha_i(j)/\bm{a}^0_i = 1$.
If $i \in \mathcal{G}(j)$, then $r_i > r_j$, so $\frac{r_j}{r_i} < 1$, and $\alpha_i(j)/\bm{a}^0_i = \frac{r_j}{r_i} < 1$. So $\alpha_i(j)/\bm{a}^0_i = 1$ when $i \in \mathcal{L}(j) \cup \mathcal{E}(j)$, and $\alpha_i(j)/\bm{a}^0_i = \frac{r_j}{r_i} < 1$ when $i \in \mathcal{G}(j)$. Because $j \in \mathcal{E}(j)$, we can attain the maximum, and therefore, $M_0(j) := \max_{ i \in [d]}\alpha_i(j)/\bm{a}^0_i = 1$.

Similarly,
\begin{align*}
\alpha_i(j)/\bm{a}^1_i :=& 
\begin{cases}
\frac{\bm{a}_i^0}{\bm{a}^1_i} = \frac{r_j \bm{a}_i^1}{{\bm{a}^1_i}}\,, & \mbox{for } i \in \mathcal{E}(j); \\ 
\frac{\bm{a}_i^0}{\bm{a}^1_i}\,, & \mbox{for } i \in \mathcal{L}(j); \\ 
\frac{r_j \bm{a}_i^1}{\bm{a}^1_i}\,, & \mbox{for } i \in \mathcal{G}(j).
\end{cases}\\[5pt]
= &
\begin{cases}
r_i = r_j\,, & \mbox{for } i \in \mathcal{E}(j); \\ 
r_i\,, & \mbox{for } i \in \mathcal{L}(j); \\ 
r_j\,, & \mbox{for } i \in \mathcal{G}(j).
\end{cases}
\end{align*}
If $i \in \mathcal{L}(j) \cup \mathcal{E}(j)$, then $r_i \leq r_j$, and $\alpha_i(j)/\bm{a}^1_i = r_i$.
If $i \in \mathcal{G}(j)$, then $r_i > r_j$, and $\alpha_i(j)/\bm{a}^1_i = r_j$. So $\alpha_i(j)/\bm{a}^1_i = r_i \leq r_j$ when $i \in \mathcal{L}(j) \cup \mathcal{E}(j)$, and $\alpha_i(j)/\bm{a}^1_i = r_j$ when $i \in \mathcal{G}(j)$. Because $j \in \mathcal{E}(j)$, we can attain the maximum, and therefore, $M_1(j) := \max_{ i \in [d]}\alpha_i(j)/\bm{a}^1_i = r_j$.

We will see that for a pair of
generalized cross polytopes, the vertical 
facets of $\D$ are described by
the linear inequalities
\begin{equation}\tag{$F(j,s)$}\label{Falpha}
\textstyle \sum_{i\in[d]} s_i \alpha_i(j) x_i + (M_0(j) - M_1(j) ) z_1 \leq M_0(j),
\end{equation}
for all $j\in[d]$ and  $s\in \{\pm1 \}^d$. 
We can also view these inequalities in the more compact nonlinear form
\begin{equation}\tag{$\bar{F}(j)$}\label{Falpha_bar}
\textstyle \sum_{i\in[d]} \alpha_i(j) |x_i| + (M_0(j) - M_1(j) ) z_1 \leq M_0(j),
\end{equation}
for all $j\in[d]$. 

\begin{example}\label{ex:2cross}
Let 
\[
\Po_0 := \{ x \in \mathbb{R}^3 ~: \pm 2 x_1 \pm 3 x_2 \pm 4 x_3 \leq 1
\},
\]
and 
\[
\Po_1 := \{ x \in \mathbb{R}^3 ~:~  \pm 5 x_1 \pm 6 x_2 \pm 7 x_3 \leq 1
\}.
\]
We have $r = (\frac{2}{5}, \frac{1}{2}, \frac{4}{7})$.
Notice that we conveniently arranged the indices $i\in[d]$ so that $r$ is sorted. 

For each $j\in[d]$, we have $2^d$ 
inequalities, 
which we calculate as follows:

\begin{itemize}[leftmargin=10.5mm]
\item[$j=1$:] $r_1 = \frac{2}{5},~ \mathcal{L}(1) = \varnothing,~ \mathcal{E}(1) = \{1\},~ \mathcal{G}(1) = \{2,3\}$. \\[2pt]
$M_0(1) = 1$, $M_1(1) = r_1= 1 \cdot \frac{2}{5} = \frac{2}{5}$.\\[2pt]
$\alpha_1(1) = \bm{a}^0_1 = 1 \cdot 2 = 2 = \frac{2}{5} \cdot 5 = r_1 \bm{a}^1_1$\,,\\[2pt]
$\alpha_2(1) = r_1 \bm{a}^1_2 = \frac{2}{5} \cdot 6 = \frac{12}{5}$,\\[2pt]
$\alpha_3(1) = r_1 \bm{a}^1_3 = \frac{2}{5} \cdot 7 = \frac{14}{5}$.\\[2pt]
Which (scaling by $5/2$) yields~   $\pm 5 x_1 \pm 6 x_2 \pm 7x_3 + \frac{3}{2} z_1 \leq \frac{5}{2}$.
\item[$j=2$:] $r_2 = \frac{1}{2},~ \mathcal{L}(2) = \{1\},~ \mathcal{E}(2) = \{2\},~ \mathcal{G}(2) = \{3\}$. \\[2pt]
$M_0(2) = 1$, $M_1(2) = r_2 = 1 \cdot \frac{3}{6} = \frac{1}{2}$.\\[2pt]
$\alpha_1(2) = \bm{a}^0_1 = 1 \cdot 2 = 2$,\\[2pt]
$\alpha_2(2) = \bm{a}^0_2 = 1 \cdot 3 = 3 = \frac{1}{2} \cdot 6 = r_2 \bm{a}^1_2$\,,\\[2pt]
$\alpha_3(2) = r_2 \bm{a}^1_3 = \frac{1}{2} \cdot 7 = \frac{7}{2}$.\\[2pt]
Which (scaling by $2$) yields~   $\pm 4 x_1 \pm 6 x_2 \pm 7x_3 + z_1 \leq 2$.
\item[$j=3$:]  $r_3 = \frac{4}{7},~ \mathcal{L}(3) = \{1,2\},~ \mathcal{E}(3) = \{3\},~ \mathcal{G}(3) = \varnothing$.\\[2pt]
$M_0(3) = 1$, $M_1(3) = r_3 = 1 \cdot \frac{4}{7} = \frac{4}{7}$.\\[2pt]
$\alpha_1(3) = \bm{a}^0_1 = 1 \cdot 2 = 2$,\\[2pt]
$\alpha_2(3) = \bm{a}^0_2 = 1 \cdot 3 = 3$,\\[2pt]
$\alpha_3(3) = \bm{a}^0_3 = 1 \cdot 4 = 4 = \frac{4}{7} \cdot 7 = r_3 \bm{a}^1_3$\,.\\[2pt]
Which yields~   $\pm 2 x_1 \pm 3 x_2 \pm 4 x_3 + \frac{3}{7} z_1 \leq 1$.
\end{itemize}
We can observe that for $j=1$ (respectively, $j=3$), the absolute values of the coefficients of the $x_i$ match the coefficients of the $x_i$ for the inequality used to define $\Po_1$ ($\Po_0$).
In contrast, for $j=2$, the coefficients of the $x_i$
are not directly tethered to the inequalities define $\Po_0$ and $\Po_1$). This establishes that we may or may not get lifting inequalities from this construction, depending on the choice of $j\in[d]$.  

Considering, for example, the non-lifting facet described by 
$4 x_1 + 6 x_2 +7x_3 + z_1 \leq 2$, 
we can see that it is satisfied as an 
equation only by the extreme points $(\frac{1}{2},0,0,0)^\top$ and 
$(0,\frac{1}{3},0,0)^\top$ of $\bar{\Po}_0$\,,
and only by the extreme points
$(0,0,\frac{1}{7},1)^\top$ and
$(0,\frac{1}{6},0,1)^\top$,
of $\bar{\Po}_1$\,. This is
a direct way to see that this inequality
describes a non-lifting facet.

We note that because
we have facets of $\D$ that are not described by lifting inequalities for
this example, we can see indirectly that
$\Po_0$ and $\Po_1$ do not have the same shape, in the sense of Theorem \ref{thm:hull_shape}.

%%%%%%%%%%%%%%%%%%%%%%
\begin{figure}
    \centering

    \begin{subfigure}{\textwidth}
        \centering
        \includegraphics[width=0.6\linewidth, trim=0.6cm 2.5cm 0.6cm 3cm, clip]{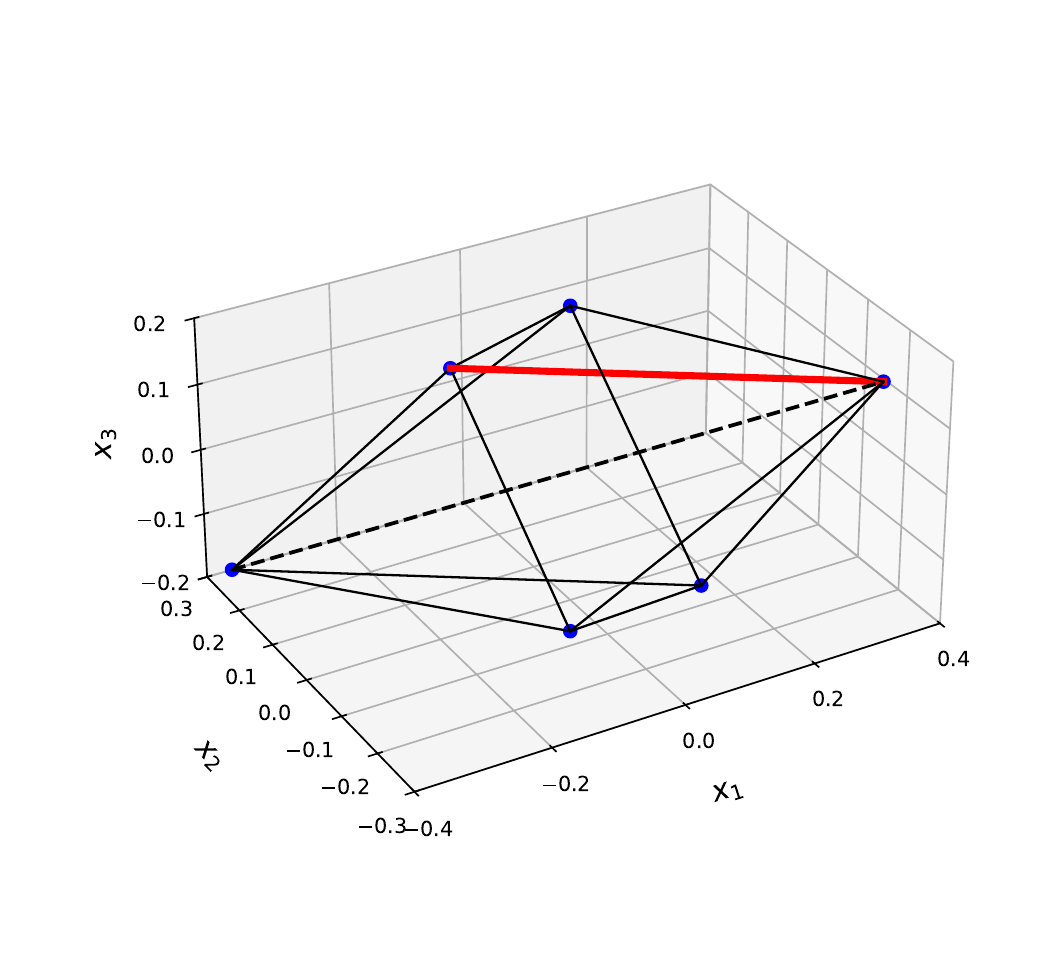}
        \caption{Slice $z_1=0$}
    \end{subfigure}

    \begin{subfigure}{\textwidth}
        \centering
        \includegraphics[width=0.6\linewidth, trim=0.6cm 2.5cm 0.6cm 3cm, clip]{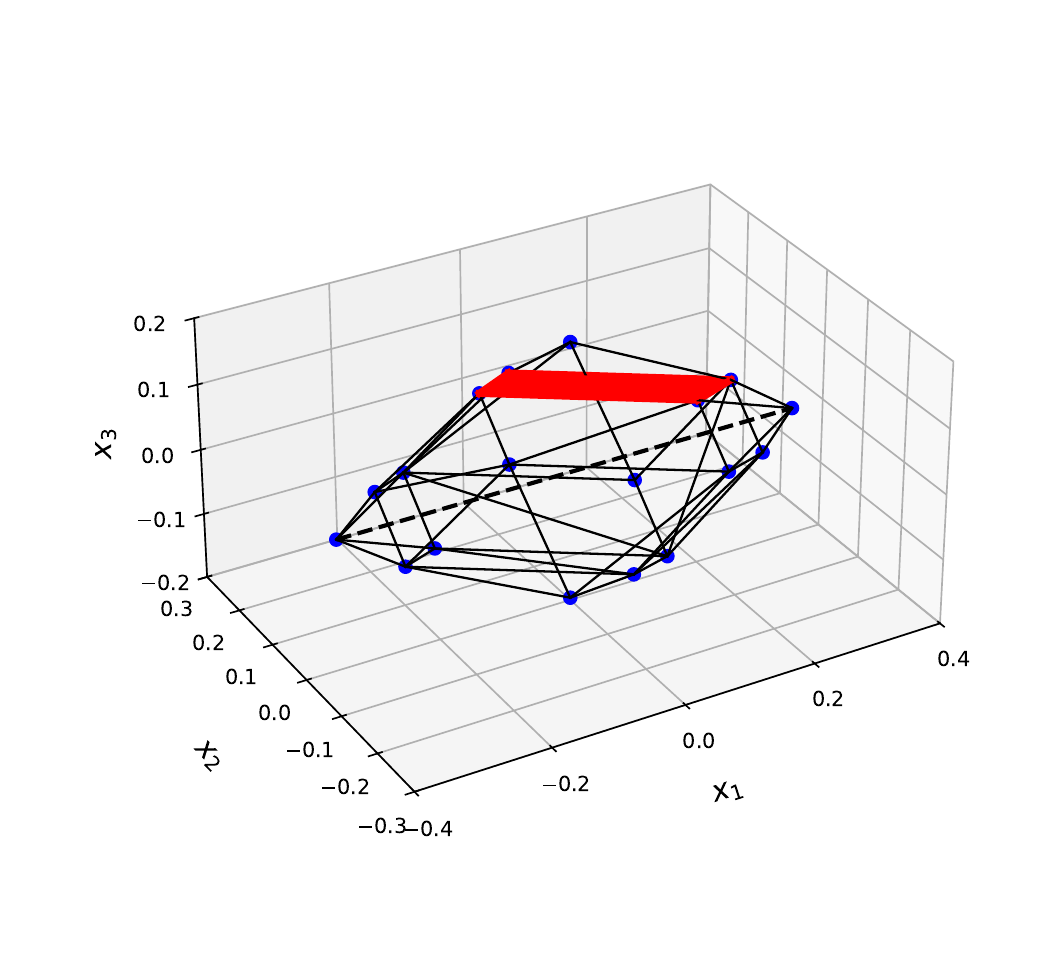}
        \caption{Slice $z_1=\frac{1}{2}$}
    \end{subfigure}

    \begin{subfigure}{\textwidth}
        \centering
        \includegraphics[width=0.6\linewidth, trim=0.6cm 2.5cm 0.6cm 3cm, clip]{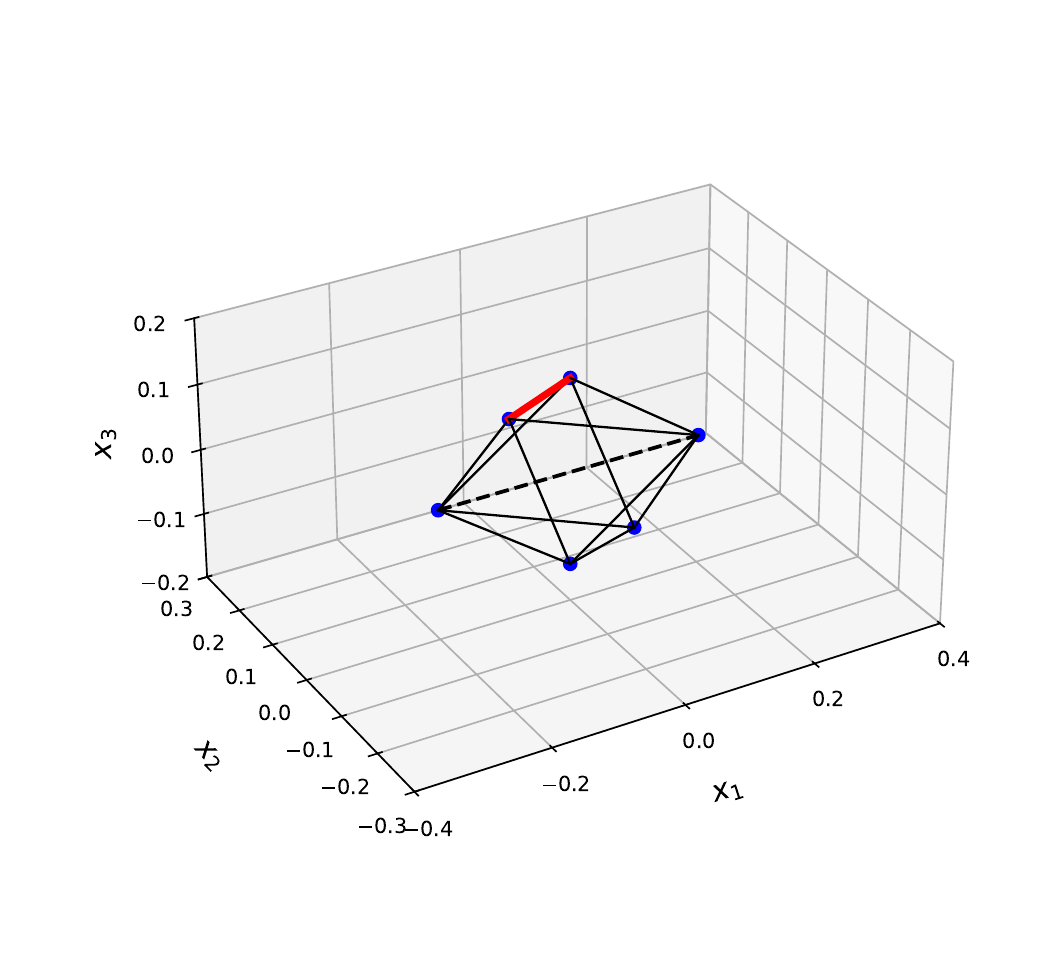}
        \caption{Slice $z_1=1$}
    \end{subfigure}
\caption{Three slices of $\mathcal{D}$ for Example \ref{ex:2cross}}\label{fig1}  
\end{figure}
%%%%%%%%%%%%%%%%%

We depict three slices of $\D$ in Figure \ref{fig1}. The first slice has 
$z_1=0$, and we recover the generalized cross polytope $\Po_0$\,.
The last slice has 
$z_1=1$, and we recover the generalized cross polytope $\Po_1$\,.
We get a geometric hint of the complexity of 
$\D$ by considering the intermediate slice with $z_1=\frac{1}{2}$.
The red in the three slices indicates the slice of the 
non-lifting facet described by  
$4 x_1 + 6 x_2 +7x_3 + z_1 \leq 2$.
We can see that it is not a lifting of a facet of 
$\Po_0$ nor $\Po_1$ because its intersection with either
hyperplane $z_1=0$ and $z_1=1$ is an edge --- hence not a facet 
--- of $\Po_0$ and of $\Po_1$\,.
\hfill $\clubsuit$
\end{example}

\begin{remark}\label{rem:sep1}
It is easy to see that for a point 
$\genfrac(){0pt}{1}{\hat x}{\hat z}\in\mathbb{R}^{d+1}$, we can solve the separation problem for the set of \ref{Falpha} 
(that is, to determine an inequality of the form
\ref{Falpha} that is violated by the given point
or establish that no such inequality exists)
in polynomial time. We consider each 
$\hat{\jmath}\in [d]$, and we set $\hat{s}_i:={\rm sign}(\hat{x}_i)$ 
 if $\hat{x}_i \neq 0$, and arbitrarily to $\pm1$ if $\hat{x}_i = 0$, 
for $i\in [d]$. If the  resulting 
$F(\hat{\jmath},\hat{s})$ is satisfied by 
$\genfrac(){0pt}{1}{\hat x}{\hat z}$, then
all of the \ref{Falpha} are satisfied by 
$\genfrac(){0pt}{1}{\hat x}{\hat z}$. If 
$F(\hat{\jmath},\hat{s})$ is violated by
$\genfrac(){0pt}{1}{\hat x}{\hat z}$, then in fact it is the most violated inequality of the 
type \ref{Falpha}.
\end{remark}

\begin{theorem}\label{thm:facet}
For every $j\in[d]$ and  $s\in \{\pm1 \}^d$, 
\ref{Falpha} describes a facet of $\mathcal{D}$.
\end{theorem}

\begin{proof} 
First, we demonstrate that \ref{Falpha} is valid for $\mathcal{D}$.
The extreme points of $\bar{\Po}_k$ are $\hat{x}^{k,i} := \left\{ \left( \genfrac{}{}{0pt}{1}{t_i \frac{1}{\bm{a}_i^k} \mathbf{e}_i}{\bm{e_k}}  \right) \right\}$, $i \in [d]$, where $t_i\in\{\pm1\}$, for $k=0,1$.
When $z_1 = 0$, we consider extreme points of $\bar{\Po}_0$ and plug into \ref{Falpha}:
\begin{align*}
&\sum_{k \in [d]} s_k \alpha_k(j) \hat{x}^{0,i}_k + (M_0(j) - M_1(j)) \cdot 0 = s_i t_i \frac{\alpha_i(j)}{\bm{a}^0_i} \leq \frac{\alpha_i(j)}{\bm{a}^0_i} \leq \max_{i \in [d]} \frac{\alpha_i(j)}{\bm{a}^0_i} = M_0(j)
\end{align*}
So, $\sum_{k\in [d]} s_k \alpha_k(j) \hat{x}^{0,i}_k  \leq M_0(j)$ holds for all extreme points of $\bar{\Po}_0$
(and in fact with equality for $i \in \mathcal{L}(j) \cup \mathcal{E}(j)$ and $s_i = t_i$). So \ref{Falpha} 
holds for all of $\bar{\Po}_0$\,. 

Similarly, when $z_1 = 1$, we consider extreme points of $\bar{\Po}_1$ and plug into \ref{Falpha}:
\begin{align*}
& \sum_{k \in [d]} s_k \alpha_k(j) \hat{x}^{1,i}_k + (M_0(j) - M
_1(j)) \cdot 1 = s_i t_i \frac{\alpha_i(j)}{\bm{a}^1_i} + M_0(j) - M_1(j) \\
& \leq  \frac{\alpha_i(j)}{\bm{a}^1_i} - M_1(j) + M_0(j) \leq \max_{i \in [d] }\frac{\alpha_i(j)}{\bm{a}^1_i} - M_1(j) + M_0(j) = M_1(j) - M_1(j) + M_0(j) = M_0(j).
\end{align*}
So $\sum_{k \in [d]} s_k \alpha_k(j) \hat{x}^{1,i}_k + M_0(j) - M_1(j) \leq M_0(j)$ 
hold for all extreme points of $\bar{\Po}_1$ 
(and in fact with equality for $i \in \mathcal{G}(j) \cup \mathcal{E}(j)$ and $s_i = t_i$). So \ref{Falpha} holds for all of $\bar{\Po}_1$\,. 

In summary, the inequalities \ref{Falpha} are valid for $\D$.

\FloatBarrier

Next, we will demonstrate that each of these inequalities
is satisfied by $d+1$ affinely-independent points of $\D$. 
    For all $i \in \mathcal{L}(j)$, we choose $v^0_i := \left( \genfrac{}{}{0pt}{1}{s_i \frac{1}{\bm{a}_i^0} \mathbf{e}_i}{0} \right)$, and plug into \ref{Falpha}:   $
s_i \alpha_i(j) \frac{s_i}{\bm{a}^0_i}+(M_0(j) - M_1(j)) 0 = s_i \bm{a}^0_i \frac{s_i}{\bm{a}^0_i} = 1 = M_0(j); 
$
    so each $v^0_i$\, for $i \in \mathcal{L}(j)$\, satisfies \ref{Falpha} as an equation. For all $i \in \mathcal{G}(j)$, we choose $v^1_i := \left( \genfrac{}{}{0pt}{1}{s_i \frac{1}{\bm{a}_i^1} \mathbf{e}_i}{1} \right)$, and 
    plug into \ref{Falpha}:
 $
s_i \alpha_i(j) \frac{s_i}{\bm{a}^1_i}+(M_0(j) - M_1(j)) = s_i r_j \bm{a}^1_i \frac{s_i}{\bm{a}^1_i} +(M_0(j) - M_1(j)) = r_j - M_1(j) + M_0(j) = M_1(j) - M_1(j) + M_0(j) = M_0(j); 
$
    so each $v^1_i$\, for $i \in \mathcal{G}(j)$ satisfies \ref{Falpha} as an equation. 
    Finally, for all $i \in \mathcal{E}(j)$, we consider the points $\left( \genfrac{}{}{0pt}{1}{s_i \frac{1}{\bm{a}_i^0} \mathbf{e}_i}{0} \right)$, and for some $\hat{\imath} \in \mathcal{E}(j)$ we pick $\left(\genfrac{}{}{0pt}{1}{s_{\hat{\imath}} \frac{1}{\bm{a}_{\hat{\imath}}^1} \mathbf{e}_{\hat{\imath}}}{1} \right)$, and those $|\mathcal{E}(j)| + 1$ points also satisfy \ref{Falpha} as an equation. Because $|\mathcal{L}(j)|+|\mathcal{G}(j)|+|\mathcal{E}(j)| = d$, we have $d+1$ points here. 

It remains to show that those $d+1$ points are affinely independent. 
We arrange these points as columns of a matrix. After appropriately 
permuting rows and columns, we arrive at the matrix 
 \[
\begin{array}{c@{\hspace{0.8em}}c}
&
\begin{array}{@{}cccc@{}}
\B{|\mathcal L(j)|} &
\B{|\mathcal G(j)|} &
\B{|\mathcal E(j)|} &
\makebox[4.8em][c]{\makebox[0pt][c]{\hspace{-0.5em}\ensuremath{1}}}
\end{array}
\\[5pt]
\begin{array}{c}
|\mathcal L(j)| \\[0.28em]
|\mathcal G(j)| \\[0.5em]
|\mathcal E(j)| \\[0.4em]
1
\end{array}
&
\left[
\begin{array}{@{}cccc@{}}
\B{D_{\mathcal L}} &
\B{\mathbf{0}} &
\B{\mathbf{0}} &
\B[4.8em]{\mathbf{0}} \\[0.25em]
\B{\mathbf{0}} &
\B{D_{\mathcal G}} &
\B{\mathbf{0}} &
\B[4.8em]{\mathbf{0}} \\[0.25em]
\B{\mathbf{0}} &
\B{\mathbf{0}} &
\B{D_{\mathcal E}} &
\B[4.8em]{\frac{s_{\hat{\imath}}}{\bm{a}_{\hat{ \imath}}^{1}}\mathbf e_{|\mathcal E(j)|}} \\[0.25em]
\B{\mathbf{0}} &
\B{\mathbf{e}^\top} &
\B{\mathbf{0}} &
\B[4.8em]{1}
\end{array}
\right],
\end{array}
\]
where $D_{\mathcal{L}} := {\rm Diag}\left( \frac{s_{i}}{\bm{a}^0_{i}}: i \in \mathcal{L}(j)\right)$, $D_{\mathcal{G}} := {\rm Diag}\left( \frac{s_{i}}{\bm{a}^1_{i}}: i \in \mathcal{G}(j)\right)$, $D_{\mathcal{E}} := {\rm Diag}\left( \frac{s_{i}}{\bm{a}^0_{i}}: i \in \mathcal{E}(j)\right)$.
Subtracting the last row multiplied
by $\frac{s_{\hat{\imath}}}{a_{\hat{ \imath}}^{1}}$ from the penultimate row, we arrive at
 \[
\begin{array}{c@{\hspace{0.8em}}c}
&
\begin{array}{@{}cccc@{}}
\B{|\mathcal L(j)|} &
\B{|\mathcal G(j)|} &
\B{|\mathcal E(j)|} &
\makebox[4.8em][c]{\makebox[0pt][c]{\hspace{-0.5em}\ensuremath{1}}}
\end{array}
\\[5pt]
\begin{array}{c}
|\mathcal L(j)| \\[0.28em]
|\mathcal G(j)| \\[0.5em]
|\mathcal E(j)| \\[0.4em]
1
\end{array}
&
\left[
\begin{array}{@{}cccc@{}}
\B{D_{\mathcal L}} &
\B{\mathbf{0}} &
\B{\mathbf{0}} &
\B[4.8em]{\mathbf{0}} \\[0.25em]
\B{\mathbf{0}} &
\B{D_{\mathcal G}} &
\B{\mathbf{0}} &
\B[4.8em]{\mathbf{0}} \\[0.25em]
\B{\mathbf{0}} &
\B{-\frac{s_{\hat{\imath}}}{\bm{a}_{\hat{\imath}}^1} \mathbf{e}_{|\mathcal{E}(j)|}\mathbf{e}^\top} &
\B{D_{\mathcal E}} &
\B[4.8em]{\mathbf{0}} \\[0.25em]
\B{\mathbf{0}} &
\B{\mathbf{e}^\top} &
\B{\mathbf{0}} &
\B[4.8em]{1}
\end{array}
\right],
\end{array}
\]
which is a lower triangular matrix with all nonzero entries on the diagonal.
Therefore, its  columns are linearly independent, and so they are also
affinely independent.
\qed
\end{proof}

\begin{theorem}\label{thm:complete}
The inequalities \ref{Falpha} for all $j\in[d]$ and $s\in \{\pm1 \}^d$, 
together with the inequalities $ 0\leq z_1 \leq 1$
give the convex hull $\D$.  
\end{theorem}

\begin{proof} 
We need only show that if $\genfrac(){0pt}{1}{\bar{x}}{\bar{z}_1}$ satisfies
 $ 0\leq z_1 \leq 1$ and \ref{Falpha} for all  $j\in[d]$ and  $s \in \{\pm1 \}^d$, then $\genfrac(){0pt}{1}{\bar{x}}{\bar{z}_1} \in \mathcal{D}$.
We establish this using the separating-hyperplane theorem. Suppose that
$\genfrac(){0pt}{1}{\bar{x}}{\bar{z}_1}$ satisfies the inequalities but is not in $\D$.
Then there exists $(p,q) \in \mathbb{R}^d \times \mathbb{R}$ such that $p^\top \bar{x} + q \bar{z}_1 > \max  
\left\{p^\top {x} + q {z}_1 ~:~ \genfrac(){0pt}{1}{x}{z} \in D\right\}$. 
Let $h_0 := h_0(p) := \max\{p^\top x ~:~ x\in\Po_0\} = \max \left\{\frac{|p_i|}{\bm{a}^0_i} ~:~ i\in [d]\right\}$, 
$h_1 := h_1(p) := \max\{p^\top x ~:~ x\in \Po_1\}= \max\left\{\frac{|p_i|}{\bm{a}^1_i} ~:~ i\in [d] \right\}$. 
With these definitions, we
have $\max\left\{p^\top {x} + q {z}_1 ~:~ \genfrac(){0pt}{1}{x}{z} \in D\right\}=\max \{h_0, h_1+q \}$.  
Clearly, $h_0, h_1 \geq 0$, and also $|p_i| \leq h_0 \bm{a}^0_i$ and $|p_i| \leq h_1 \bm{a}^1_i$\,, for $i\in[d]$.

Next, we wish to show that without loss of generality, we can take $q=h_0-h_1$\,.
Formally, we 
will check that if there exists $(p,q) \in \mathbb{R}^d \times \mathbb{R}$ such that $p^\top \bar{x} + q \bar{z}_1 > \max  
\left\{p^\top {x} + q {z}_1 ~:~ \genfrac(){0pt}{1}{x}{z} \in D\right\}$,
then there exist $p \in \mathbb{R}^d$ such that
$p^\top \bar{x} + q^* \bar{z}_1 > \max  
\left\{p^\top {x} + q^* {z}_1 ~:~ \genfrac(){0pt}{1}{x}{z} \in D\right\}$,
where $q^* := h_0 - h_1$\,.
We rewrite $p^\top \bar{x} + q \bar{z}_1 > \max \{h_0, h_1 + q\}$ as $p^\top \bar{x} > \max \{h_0, h_1 + q\} - q \bar{z}_1$. Let $g(y):= \max \{h_0  - y \bar{z}_1\,, h_1 + (1 - \bar{z}_1)y \}$. We observe that $h_0  - y \bar{z}_1$ has nonpositive slope $-\bar{z}_1$\,, and $h_1 + (1 - \bar{z}_1)y$ has nonnegative slope $1 - \bar{z}_1$\,. The point-wise maximum of affine functions is convex, so $g(y)$ is convex and minimized at the breakpoint $y = h_0 - h_1$\,.
Finally, we conclude that $p^\top \bar{x} + q^* \bar{z}_1 > \max \{h_0, h_1 + q^* \}$.

When $q^* = h_0 - h_1$\,, 
then $\max \{ h_0, h_1+q^* \} = h_0 = h_1 + q^*$. So we get $p^\top \bar{x} + q^* \bar{z}_1 > h_0 \Rightarrow p^\top \bar{x} + (h_0-h_1) \bar{z}_1 > h_0 \Rightarrow p^\top \bar{x} > (1-\bar{z}_1)h_0 + \bar{z}_1 h_1$\,.

The maximum value of $z_1$ on both $\D$ and feasible region defined by the inequalities is the same, which is $1$. Therefore, $p \neq 0$.
Because $\Po_k$ is centrally symmetric and full dimensional, we can see that $h_k > 0$ for $k=0,1$.

We define $L(r) := \sum_{i\in [d]}\min \{ \bm{a}^0_i, r \bm{a}^1_i \} | \bar{x}_i|$, $R(r) := (1-\bar{z}_1) + \bar{z}_1 r$, and $\phi(r) := L(r) - R(r)$. We note that $R(\cdot)$ is an affine function and $L(\cdot)$ is a concave piecewise-linear function, so $\phi(\cdot)$ is a concave piecewise-linear function.
 We will eventually gain a contradiction by demonstrating that there
 is  a $\hat{\jmath} \in [d]$ such that $\phi(r_{\hat{\jmath}}) > 0$, and 
 so $\genfrac(){0pt}{1}{\bar{x}}{\bar{z}}$ violates $F(\hat{\jmath},s)$ for some $s\in\{\pm1\}^d$.

Toward this end, we define $r^* := h_1/h_0$\,, and 
we will first
 demonstrate that  $\phi(r^*) > 0$.
 Because $|p_i| \leq h_0 \bm{a}^0_i$ and $|p_i| \leq h_1 \bm{a}^1_i = h_0 r^* \bm{a}^1_i$ for all $i \in [d]$, $\sum_{i\in[d]} |p_i| |\bar{x}_i| \leq h_0 \sum_{i\in [d]} \min \{ \bm{a}^0_i, r^* \bm{a}^1_i \} | \bar{x}_i| = h_0 L(r^*)$.  
Then, $h_0 R(r^*) = (1-\bar{z}_1)h_0 + \bar{z}_1h_0 r^* = (1-\bar{z}_1)h_0 + \bar{z}_1 h_1 < p^\top \bar{x} \leq \sum_{i \in [d]} |p_i| |\bar{x}_i| \leq h_0 L(r^*)$, and thus $h_0 \phi(r^*) = h_0 (L(r^*) - R(r^*)) > 0$. Because $h_0> 0, \phi(r^*) > 0$.

Now, we are ready to identify an index $\hat{\jmath} \in [d]$ such that $\phi(r_{\hat{\jmath}}) > 0$ (in case there is no $j \in [d]$ with $r_j = r^*$).
 Choose $i_0 \in \argmax \left\{\frac{|p_i|}{\bm{a}^0_i} ~:~ i\in[d]\right\}$, $i_1 \in \argmax \left\{\frac{|p_i|}{\bm{a}^1_i} ~:~ i\in[d]\right\}$. 
 $h_1 = \frac{|p_{i_1}|}{\bm{a}^1_{i_1}} \leq \frac{h_0 \bm{a}^0_{i_1}}{\bm{a}^1_{i_1}} = h_0 r_{i_1} \Rightarrow r_{i_1} \geq \frac{h_1}{h_0} \Rightarrow r_{i_1} \geq r^*$. Also, because $h_0 = \frac{|p_{i_0}|}{\bm{a}^0_{i_0}} \Rightarrow |p_{i_0}| = h_0 \bm{a}^0_{i_0}$ and $h_1 \geq \frac{|p_{i_0}|}{\bm{a}^1_{i_0}} \Rightarrow |p_{i_0}| \leq h_1 \bm{a}^1_{i_0}$, $h_0 \bm{a}^0_{i_0} \leq h_1 \bm{a}^1_{i_0} \Rightarrow \frac{\bm{a}^0_{i_0}}{\bm{a}^1_{i_0}} \leq \frac{h_1}{h_0} \Rightarrow r_{i_0} \leq r^*$. So we can define $r^- := \max_{i \in [d]} \{r_i: r_i \leq r^*\}$ and $r^+ := \min_{i \in [d]} \{r_i: r_i \geq r^*\}$. 
 When $r_i \leq r$, we have $\min \{ \bm{a}^0_i, r \bm{a}^1_i \} = \bm{a}^0_i$, and 
 when $r_i \geq r$, we have $\min \{ \bm{a}^0_i, r \bm{a}^1_i \} = r \bm{a}^1_i$\,, and because $L(\cdot)$ has no breakpoint between $r^-$ and $r^+$, and therefore
 $L(\cdot)$  and hence  $\phi(\cdot)$ are affine functions on $[r^-,r^+]$. Equivalently, because $h_0 > 0$, the function $h_0\phi(\cdot)$ is also affine function on $[r^-,r^+]$. Therefore, we have $\max \{ h_0 \phi(r^-), h_0 \phi(r^+)\} \geq h_0 \phi(r^*) > 0$. Thus, $h_0 \phi(r^-) > 0$ or $h_0 \phi(r^+) > 0$. Because $h_0 > 0$, we have $\phi(r^-) > 0$ or $\phi(r^+) > 0$. Because $r^-$ and $r^+$ both are values of $r_i$ for some $i \in [d]$, there exists $\hat{\jmath} \in [d]$ such that $\phi(r_{\hat{\jmath}}) > 0$. 

For this $\hat{\jmath}$, we define 
\[
\alpha_i(\hat{\jmath}):= 
\begin{cases}
\bm{a}_i^0\,, & r_i \leq r_{\hat{\jmath}}\,; \\ 
r_{\hat{\jmath}} \bm{a}_{i}^1 \,, & r_i \geq r_{\hat{\jmath}}\,.
\end{cases}
\]
Then, by definition, we have $M_0(\hat{\jmath}) := \max_{ i \in [d]}\alpha_i(\hat{\jmath})/\bm{a}^0_i = 1$\,, $M_1(\hat{\jmath}) := \max_{ i \in [d]}\alpha_i(\hat{\jmath})/\bm{a}^1_i = r_{\hat{\jmath}}$\,.

For $i \in [d]$,
let $\bar{s}_i := {\rm sign}(\bar{x}_i)$ if $\bar{x}_i \neq 0$, 
and choose $\bar{s}_i := \pm 1$ 
arbitrarily if $\bar{x}_i = 0$.
Then $\sum_{i\in[d]}\bar{s}_i\alpha_i(\hat{\jmath})\bar{x}_i = \sum_{i\in[d]} \alpha_i(\hat{\jmath}) |\bar{x}_i| \geq L(r_{\hat{\jmath}})$. $ R(r_{\hat{\jmath}}) = (1-\bar{z}_1) + \bar{z}_1 r_{\hat{\jmath}}$\,. Because $\phi(r_{\hat{\jmath}}) > 0, L(r_{\hat{\jmath}})-R(r_{\hat{\jmath}})>0 \Rightarrow \sum_{i \in [d]} \bar{s}_i\alpha_i(\hat{\jmath})\bar{x}_i > (1-\bar{z}_1) + \bar{z}_1 r_{\hat{\jmath}}$\,. Using $M_0(\hat{\jmath}) = 1, M_1(\hat{\jmath}) = r_{\hat{\jmath}}$\,, finally we have $\sum_{i \in [d]} \bar{s}_i\alpha_i(\hat{\jmath})x_i + (M_0(\hat{\jmath}) - M_1(\hat{\jmath}))\bar{z}_1 > M_0(\hat{\jmath})$. But then $\genfrac(){0pt}{1}{\bar{x}}{\bar{z}}$ violates $F(\hat{\jmath},\bar{s})$, a contradiction. 
\qed
\end{proof}

\begin{theorem}\label{thm:ratiosettheorem}
    We define the \emph{ratio set} $R := \{r_i: i \in [d]\}$. 
    All vertical facets of $\mathcal{D}$ are described by optimal big-M lifting inequalities if and only if $|R| \leq 2$.
\end{theorem}

\begin{proof} 
Let $r_{\min} := \min\{r_i: i \in [d]\}$, $r_{\max} := \max\{r_i: i \in [d]\}$. Also, let $S^0(r_{{j}}):= \{ i \in [d]: r_i \leq r_{{j}}\}$
and $S^1(r_{{j}}):= \{ i \in [d]: r_i \geq r_{{j}} \}$,
for all $j \in [d]$. 

$(\Rightarrow):$ 
Suppose that $|R| \geq 3$. Choose $r_{\hat{\jmath}} \in R$ such that $r_{\min} < r_{\hat{\jmath}} < r_{\max}$\,. 
Note that 
$|S^0(r_{\hat{\jmath}})| \leq d-1$ and $|S^1(r_{\hat{\jmath}})| \leq d-1$. 
This implies that any set of $d+1$ extreme
points of $\D$ satisfying the inequality as 
an equation cannot have $d$ points from
either of $\bar{\Po}_0$ or $\bar{\Po}_1$\,.
Therefore, the inequality cannot be a lifting of a facet-describing inequality of $\bar{\Po}_0$ or $\bar{\Po}_1$\,.

$(\Leftarrow):$ Suppose that $|R| = 1$, say $r_i = r$ for all $i\in[d]$. Then, $S^0(r_i) = S^1(r_i) = [d], \mbox{ for all } i \in [d]$. We can choose $d$ affinely-independent points from $\Po_0$ and one extreme point from $\Po_1$\,. Similarly, we can choose $d$ affinely-independent points from $\Po_1$ and one extreme point from $\Po_0$\,.
And in both cases, every such facet is described by an optimal big-M lifting inequality. 
In fact, it is easy to see that in this case ($|R|=1$) that Theorem \ref{thm:hull_shape} applies. 

Suppose that $|R| = 2$, that is $R = \{r_{\min}, r_{\max}\}$ with $r_{\min} < r_{\max}$\,. Then, we have two cases:
\begin{enumerate}[wide, labelwidth=0pt, labelindent=0pt]
    \item $S^0(r_{\min}) = \{i: r_i = r_{\min} \} \neq \varnothing$, and $S^1(r_{\min}) = \{i: r_i \geq r_{\min} \} = [d]$. Then, we choose $d$ affinely-independent points $\left( \genfrac{}{}{0pt}{1}{ \frac{s_i}{\bm{a}_i^1} \mathbf{e}_i}{1} \right)$ for $i \in S^1(r_{\min})$, and one more point $\left( \genfrac{}{}{0pt}{1}{ \frac{s_i}{\bm{a}_i^0} \mathbf{e}_i}{0} \right)$ for $i \in S^0(r_{\min})$. So \ref{Falpha} with $j \in S^0(r_{\min}) = \{i: r_i =r_{\min}\}$
    is an optimal big-M lifting inequality starting from $\Po_{k_2}$\,. 
    \item $S^0(r_{\max}) = \{i: r_i \leq r_{\max} \} = [d]$, and $S^1(r_{\max}) = \{i: r_i = r_{\max} \} \neq \varnothing$. Then, we choose $d$ affinely-independent points $\left( \genfrac{}{}{0pt}{1}{ \frac{s_i}{\bm{a}_i^0} \mathbf{e}_i}{0} \right)$ for $i \in S^0(r_{\max})$, and one more point $\left( \genfrac{}{}{0pt}{1}{ \frac{s_i}{\bm{a}_i^1} \mathbf{e}_i}{1} \right)$ for $i \in S^1(r_{\max})$. So \ref{Falpha} with $j \in S^1(r_{\max}) = \{i: r_i =r_{\max}\}$  is an optimal big-M lifting inequality starting from $\Po_{k_1}$\,. \qed 
\end{enumerate}  
\end{proof}
Considering the proof of Theorem
\ref{thm:ratiosettheorem}, we can easily see the following.
\begin{observation}
 $F(\hat{\jmath},s)$ is an optimal big-M lifting inequality starting from a facet-describing inequality of $\Po_0$ 
 (resp., $\Po_1$)
if and only if $r_{\hat{\jmath}} =r_{\max}$
 (resp., $r_{\hat{\jmath}} =r_{\min}$).
\end{observation}

Considering again Example \ref{ex:2cross}
in light of Theorem \ref{thm:ratiosettheorem},
we have $|R|=3$, and this
is why we see non-lifting inequalities in the complete description.

In the following example also with $d=3$, we have $|R|=2$,
and we only see lifting inequalities. 
We note that this example, where all vertical facets are described by optimal big-M lifting inequalities
is not covered by
Theorem \ref{thm:hull_low} nor
Theorem \ref{thm:hull_shape}, demonstrates that Theorem 
\ref{thm:ratiosettheorem} really covers new cases. 

\begin{example}
Let 
\[
\Po_0 := \{ x \in \mathbb{R}^3 ~:~  \pm 2 x_1 \pm 3 x_2 \pm 4 x_3 \leq 1
\},
\]
and let 
\[
\Po_1 := \{ x \in \mathbb{R}^3 ~:~  \pm 5 x_1 \pm 6 x_2 \pm 8 x_3 \leq 1
\}.
\]

\begin{itemize}[leftmargin=10.5mm]
\item[$j=1$:] $r_1 = \frac{2}{5},~ \mathcal{L}(1) = \varnothing,~ \mathcal{E}(1) = \{1\},~ \mathcal{G}(1) = \{2,3\}$. \\[2pt]
$M_0(1) = 1$, $M_1(1) = r_1= 1 \cdot \frac{2}{5} = \frac{2}{5}$.\\[2pt]
$\alpha_1(1) = \bm{a}^0_1 = 1 \cdot 2 = 2 = \frac{2}{5} \cdot 5 = r_1 \bm{a}^1_1$\,,\\[2pt]
$\alpha_2(1) = r_1 \bm{a}^1_2 = \frac{2}{5} \cdot 6 = \frac{12}{5}$,\\[2pt]
$\alpha_3(1) = r_1 \bm{a}^1_3 = \frac{2}{5} \cdot 8 = \frac{16}{5}$.\\[2pt]
Which (scaling by $\frac{5}{2}$) yields~   $\pm 5 x_1 \pm 6 x_2 \pm 8 x_3 + \frac{3}{2} z_1 \leq \frac{5}{2}$.
\item[$j=2$:] $r_2 = \frac{1}{2},~ \mathcal{L}(2) = \{1\},~ \mathcal{E}(2) = \{2, 3\},~ \mathcal{G}(2) = \varnothing$. \\[2pt]
$M_0(2) = 1$, $M_1(2) = r_2 = 1 \cdot \frac{3}{6} = \frac{1}{2}$.\\[2pt]
$\alpha_1(2) = \bm{a}^0_1 = 1 \cdot 2 = 2$,\\[2pt]
$\alpha_2(2) = \bm{a}^0_2 = 1 \cdot 3 = 3 = \frac{1}{2} \cdot 6 = r_2 \bm{a}^1_2$\,,\\[2pt]
$\alpha_3(2) = \bm{a}^0_3 = 1 \cdot 4 = 4 = \frac{1}{2} \cdot 8 = r_2 \bm{a}^1_3$\,.\\[2pt]
Which yields~   $\pm 2 x_1 \pm 3 x_2 \pm 4 x_3 + \frac{1}{2} z_1 \leq 1$.
\item[$j=3$:]  $r_3 = \frac{1}{2},~ \mathcal{L}(3) = \{1\},~ \mathcal{E}(3) = \{2,3\},~ \mathcal{G}(3) = \varnothing$.\\[2pt]
$M_0(3) = 1$, $M_1(3) = r_3 = 1 \cdot \frac{4}{8} = \frac{1}{2}$.\\[2pt]
$\alpha_1(3) = \bm{a}^0_1 = 1 \cdot 2 = 2$,\\[2pt]
$\alpha_2(3) = \bm{a}^0_2 = 1 \cdot 3 = 3 = \frac{1}{2} \cdot 6 = r_3 \bm{a}^1_2$\,,\\[2pt]
$\alpha_3(3) = \bm{a}^0_3 = 1 \cdot 4 = 4 = \frac{1}{2} \cdot 8 = r_3 \bm{a}^1_3$\,.\\[2pt]
Which yields~   $\pm 2 x_1 \pm 3 x_2 \pm 4 x_3 + \frac{1}{2} z_1 \leq 1$.
\end{itemize}
In all  cases, we can see that the inequalities
produced are lifting inequalities. \hfill $\clubsuit$
\end{example}

%%%%%%%%%%%%%%%%%%%

\subsection{\texorpdfstring{General $n\geq 1$}{General n>=1}}\label{sec:generaln}

For more than two generalized cross polytopes, we need to generalize the
inequalities \ref{Falpha}.
Consider $k_1,k_2 \in N_0$ and $k_1 < k_2$\,. For each $j \in [d]$, we define $r^{k_1,k_2}_j := \bm{a}^{k_1}_j /\bm{a}^{k_2}_j$\,,
we partition $[d]$, depending on the $r^{k_1,k_2}_j$\,, into the blocks:
\begin{align*}
 \mathcal{L}^{k_1,k_2}(j) &:= \left\{i \in [d]: r^{k_1,k_2}_i < r^{k_1,k_2}_j\right\}, \\
 \mathcal{E}^{k_1,k_2}(j) &:= \left\{i \in [d]: r^{k_1,k_2}_i = r^{k_1,k_2}_j\right\}, \\
 \mathcal{G}^{k_1,k_2}(j) &:= \left\{i \in [d]: r^{k_1,k_2}_i > r^{k_1,k_2}_j\right\}. 
\end{align*}
 Next, for each $i\in[d]$, we define 
\[
\alpha^{k_1,k_2}_i(j):= 
\begin{cases}
\bm{a}_i^{k_1}=  r^{k_1,k_2}_j \bm{a}_i^{k_2}\,, & \mbox{for } i \in \mathcal{E}^{k_1,k_2}(j); \\ 
 \bm{a}_i^{k_1}\,, & \mbox{for } i \in \mathcal{L}^{k_1,k_2}(j); \\ 
 r^{k_1,k_2}_j \bm{a}_i^{k_2}\,, & \mbox{for } i \in \mathcal{G}^{k_1,k_2}(j).
\end{cases}
\]
Equivalently, $\alpha^{k_1,k_2}_i(j):= \min \{ \bm{a}_i^{k_1},  r^{k_1,k_2}_j \bm{a}_i^{k_2}\}$. Finally, for each $k \in N_0$, we define $M^{k_1,k_2}_k(j) := \max_{ i \in [d]}\alpha^{k_1,k_2}_i(j)/\bm{a}^k_i $\,.

Now,
$k_1,k_2 \in N_0$\,, with $k_1 < k_2$\,,  $j\in[d]$, and  $s\in \{\pm1 \}^d$, 
we define 
\begin{equation}\tag{$F(k_1,k_2,j,s)$}\label{Fkalpha}
\sum_{i\in [d]} s_i \alpha^{k_1,k_2}_i(j) x_i + \sum_{k\in N} \left(M^{k_1,k_2}_0(j) - M^{k_1,k_2}_k(j) \right) z_k \leq M^{k_1,k_2}_0(j).
\end{equation}
We can also view these inequalities in the more compact nonlinear form
\begin{equation}\tag{$\bar{F}(k_1,k_2,j)$}\label{Fkalpha_bar}
\sum_{i\in [d]}  \alpha^{k_1,k_2}_i(j) |x_i| + \sum_{k\in N} \left(M^{k_1,k_2}_0(j) - M^{k_1,k_2}_k(j) \right) z_k \leq M^{k_1,k_2}_0(j),
\end{equation}
for all $j\in[d]$. 

\begin{observation}
For $n:=1$, we have uniquely $k_1=0$ and $k_2=1$, and then 
    \ref{Fkalpha} (resp., \ref{Fkalpha_bar})
    reduces to \ref{Falpha} (resp., \ref{Falpha_bar}).
\end{observation}

Next, we generalize Remark \ref{rem:sep1}, Theorem \ref{thm:facet}, and Theorem \ref{thm:ratiosettheorem} to the case of $n>1$. 

\begin{remark}\label{rem:sep_gen}
It is easy to see that for a point 
$\genfrac(){0pt}{1}{\hat x}{\hat z}\in\mathbb{R}^{d+n}$, we can solve the separation problem for the set of \ref{Fkalpha} 
in polynomial time. For each pair $k_1,k_2 \in N_0$ with $k_1 < k_2$, and for each $\hat{\jmath} \in [d]$, and we set $\hat{s}_i:={\rm sign}(\hat{x}_i)$ 
 if $\hat{x}_i \neq 0$, and arbitrarily to $\pm1$ if $\hat{x}_i = 0$, 
for $i\in [d]$. If the  resulting 
$F(k_1, k_2, \hat{\jmath},\hat{s})$ is satisfied by 
$\genfrac(){0pt}{1}{\hat x}{\hat z}$, then
all of the \ref{Fkalpha} are satisfied by 
$\genfrac(){0pt}{1}{\hat x}{\hat z}$. If 
$F(k_1, k_2, \hat{\jmath},\hat{s})$ is violated by
$\genfrac(){0pt}{1}{\hat x}{\hat z}$, then in fact it is the most violated inequality of the 
type \ref{Fkalpha}. Therefore, by checking  
at most $\binom{n+1}{2}d$ choices (note that 
for some $j$, the $r^{k_1,k_2}_j$ need not all be distinct),
we can find the most violated inequality or establish that there
are none violated in polynomial time.
\end{remark}

\begin{theorem}\label{thm:facets_general_n}
 (proof in Appendix \ref{app:facet})
For every  $k_1,k_2 \in N_0$\,, with $k_1 < k_2$\,,  $j\in[d]$, and  $s\in \{\pm1 \}^d$, 
\ref{Fkalpha} describes a facet of $\mathcal{D}$.
\end{theorem}

\begin{theorem}\label{thm:R_generaln} 
 (proof in Appendix \ref{app:lift})
 For $k_1,k_2 \in N_0$\,, with $k_1 < k_2$\,,  $j\in[d]$, and  $s\in \{\pm1 \}^d$, 
the facet of $\D$ described by 
\ref{Fkalpha} is a lifting inequality if and only if 
$|R^{k_1,k_2}| \leq 2$, where we define the \emph{ratio set} $R^{k_1,k_2}:= \left\{r^{k_1,k_2}_i: i \in [d]\right\}$.
\end{theorem}

Finally, we demonstrate that 
not every vertical facet is 
an \ref{Fkalpha} for $n=2$.

\begin{example}\label{ex:tree}
With $n=2$ and $d=4$, we 
consider
    \[\Po_0 := \{ x \in \mathbb{R}^4 ~: \pm x_1 \pm x_2 \pm 10 x_3 \pm 10 x_4 \leq 1\},\]
    \[\Po_1 :=\textstyle \{ x \in \mathbb{R}^4 ~: \pm 10 x_1 \pm \frac{1}{2} x_2 \pm \frac{1}{2} x_3 \pm 10 x_4 \leq 1\},\]
    \[\Po_2 :=\textstyle \{ x \in \mathbb{R}^4 ~: \pm 10 x_1 \pm 10 x_2 \pm \frac{1}{3} x_3 \pm \frac{1}{3} x_4 \leq 1\}.\] 
Consider the inequality 
\begin{equation}\label{star}\tag{$*$}
x_1+x_2+x_3+x_4-z_1-2z_2 \leq 1.
\end{equation}

In what follows, all points are in $\mathbb{R}^4\times \mathbb{R}^2$.

\noindent The extreme points of $\bar{\Po_0}$ are 
$\left( \genfrac{}{}{0pt}{1}{\pm \mathbf{e}_1}{\mathbf{e}_0}\right),
\left( \genfrac{}{} {0pt}{1}{\pm \mathbf{e}_2}{\mathbf{e}_0}\right), \left( \genfrac{}{}{0pt}{1}{\pm \frac{1}{10}\mathbf{e}_3}{\mathbf{e}_0}\right),
\left( \genfrac{}{}{0pt}{1}{\pm \frac{1}{10}\mathbf{e}_4}{\mathbf{e}_0}\right)$. 

\noindent The extreme points of $\bar{\Po_1}$ are $\left( \genfrac{}{}{0pt}{1}{\pm \frac{1}{10} \mathbf{e}_1}{\mathbf{e}_1}\right),\left( \genfrac{}{}{0pt}{1}{\pm 2 \mathbf{e}_2}{\mathbf{e}_1}\right), \left( \genfrac{}{}{0pt}{1}{\pm 2\mathbf{e}_3}{\mathbf{e}_1}\right), \left( \genfrac{}{}{0pt}{1}{\pm \frac{1}{10}\mathbf{e}_4}{\mathbf{e}_1}\right)$. 

\noindent The extreme points of $\bar{\Po_2}$ are $\left( \genfrac{}{}{0pt}{1}{\pm \frac{1}{10} \mathbf{e}_1}{\mathbf{e}_2}\right),\left( \genfrac{}{}{0pt}{1}{\pm \frac{1}{10} \mathbf{e}_2}{\mathbf{e}_2}\right), \left( \genfrac{}{}{0pt}{1}{\pm 3\mathbf{e}_3}{\mathbf{e}_2}\right), \left( \genfrac{}{}{0pt}{1}{\pm 3\mathbf{e}_4}{\mathbf{e}_2}\right)$. 

It is easy to check that \eqref{star} is valid for all of these extreme point, and hence also for $\D$. Consider the $n+d=6$ particular extreme points of $\D$: $\left( \genfrac{}{}{0pt}{1}{\mathbf{e}_1}{\mathbf{e}_0}\right),\left( \genfrac{}{}{0pt}{1}{ \mathbf{e}_2}{\mathbf{e}_0}\right), \left( \genfrac{}{}{0pt}{1}{ 2 \mathbf{e}_2}{\mathbf{e}_1}\right), \left( \genfrac{}{}{0pt}{1}{ 2\mathbf{e}_3}{\mathbf{e}_1}\right), \left( \genfrac{}{}{0pt}{1}{ 3\mathbf{e}_3}{\mathbf{e}_2}\right), \left( \genfrac{}{}{0pt}{1}{ 3\mathbf{e}_4}{\mathbf{e}_2}\right)$. It is straightforward to check that these $6$ points satisfy \eqref{star} as an equation and that they are affinely independent. However, the inequality is not in the form of \ref{Fkalpha} --- see Appendix \ref{app:ex} for details.
 \hfill $\clubsuit$
\end{example}

%%%%%%%%%%%%%%%%%%%

\section{Experiments}\label{sec:exper}

 We wanted to get some idea about how useful 
our inequalities are for situations that might be close to how they could arise in practice. So, we designed some experiments with
 $3\leq d \leq 12$ and $2\leq n \leq 20$, using \emph{translated} generalized cross polytopes. 
 Further, we
wanted to see how we could systematically improve them in 
situations where there might be some side constraints, as would be typical in applications.

We ran our experiments on our server ``zebratoo'' (running Windows Server 2022 Standard):
two Intel Xeon Gold 6444Y 3.60GHz processors, with 16 cores each, and 128 GB of memory. 
We carried out the experiments
using \texttt{CVXPY}, an open source Python-embedded modeling language for convex-optimization problems, which allows easy use of convex absolute-value constraints.
For the solver, we use \texttt{CLARABEL}, which is a conic optimization solver
distributed with  \texttt{CVXPY}.
 We found that \texttt{CLARABEL}
    is a stable and efficient solver that could handle our formulations. 

\subsection{Instance generation}
We aimed to generate 
challenging instances as follows:

\begin{enumerate}[leftmargin=*]
    \item For \(k\in N_0\)\,, we sample $\bm{a}^k\in\mathbb{R}_{++}^d$ so that for every pair \(k_1<k_2\), the ratios \(\{\bm{a}_i^{k_1}/\bm{a}_i^{k_2} ~:~ i\in[d]\}\) are all distinct. 
    By having many distinct ratios, we can expect that there will be many facets for $\D$ that are not described by liftings of facet-describing inequalities of the $\Po_i$\,.
    \item For \(k\in N_0\)\,, we uniformly sample the translation vector \(t^k\in\mathbb{R}^d\) on the unit sphere, aiming to spread out the $\Po_i$\,,
    as would be more typical in applications than having them all centered at the origin.

    Thus, for \(k\in N_0\)\,,
\[
\textstyle
\Po_k:=\left\{x\in\mathbb{R}^d ~:~ \sum_{i\in[d]} \bm{a}_i^k |x_i-t^k_{i}|\le 1\right\}.
\]

    \item We randomly sample \(c\in\mathbb{R}^d\) as an integer vector from $[-d,d]^d$, compute 
    \[
    v_k:=\max\{c^\top x ~:~ x\in \Po_k\} ,~ k \in N_0\,,
    \]
    and set
    \[
    g_k:=v_0-v_k\,, \forall \, k \in N_0\,,  
    \]
so that the objectives \(g_k+c^\top x\) are competitive with each other. 
\item Next, we added two side constraints to each $\Po_k$\, by trimming off a best vertex,
measured by $c^\top x$.
We also trim off its reflection (which is a worst vertex), 
as a convenience. Specifically,
for each $k\in N_0$\,, let
\[
i_k \in \argmax_{i\in[d]} \frac{|c_i|}{\bm{a}_i^k},
\]
Thus, \(i_k\) is the index for the best/worst vertex $\pm\hat{x}^{k,i}$ for each \(\Po_k\)\,. 
Finally, we define
\[
\Po_k^+
:=
\Po_k
\cap
\left\{
x :
\left|x_{i_k}-t^k_{i_k}\right|
\leq
\frac{1}{2 \bm{a}^k_{i_k}}
\right\}.
\]
That is, for each $k\in N_0$\,, we cut the best vertex and its reflection, midway between the vertex and the center $t^k$\,.
\end{enumerate}

\begin{figure}%[t!]
    \centering
\includegraphics[width=0.35\linewidth]{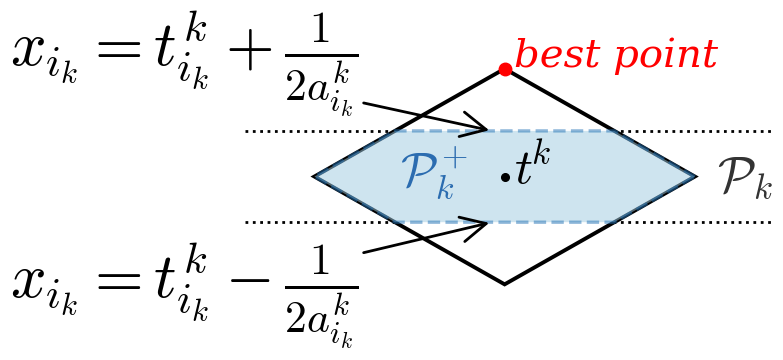}
    \caption{Trimmed polytope}
    \label{fig:trimmeddiamond}
\end{figure}

\subsection{Relaxations}

We consider six relaxations, $\mathcal{F}$, $\mathcal{L}$, $\mathcal{F}^+$, $\mathcal{L}^+$, $\mathcal{F}^{++}$, $\mathcal{L}^{++}$, which we now describe.

\begin{itemize}[leftmargin=*]
\item Relaxations $\mathcal{F}$ and $\mathcal{L}$ (which ignore that we cut off a pair of vertices from each $\Po_k$).
\begin{itemize}
    \item Relaxation $\mathcal{F}$:
Use inequalities \ref{Fkalpha_bar} for every $j$ and every pair \((k_1,k_2)\), and relax the indicator variables continuously:
\[
z_k \geq 0 \quad \forall k\in N, 
\qquad
\textstyle \sum_{k\in N} z_k \le 1.
\]

\item Relaxation $\mathcal{L}$: 
Keep only the inequalities \ref{Fkalpha_bar} corresponding to
\[
r_j^{k_1,k_2}=r_{\min}^{k_1,k_2}
\qquad\text{or}\qquad
r_j^{k_1,k_2}=r_{\max}^{k_1,k_2}
\]
(such inequalities correspond to full optimal big-M liftings). 
Also, \(z\) is relaxed continuously:
\[
z_k \geq 0 \quad \forall k\in N, 
\qquad
\textstyle \sum_{k\in N} z_k \le 1.
\]
\end{itemize}
Comparing $\mathcal{F}$ and $\mathcal{L}$,
we can learn something about the importance of our non-lifting facets. 

\item From \(\mathcal{F},\mathcal{L}\) to \(\mathcal{F}^+,\mathcal{L}^+\).

Improve the lifting coefficients based on $\Po_k^+$, $\forall k \in N$:
\[
\begin{aligned}
&M^{k_1,k_2}_k(j)
:=
\max_{x\in \Po_k^+}\sum_{i \in [d]}\alpha^{k_1,k_2}_i(j) |x_i-t^k_i|, 
\\ \qquad
&\mu^{k_1,k_2}_k(j)
:=
M^{k_1,k_2}_0(j)
-
M^{k_1,k_2}_k(j) , \forall k \in N.
\end{aligned}
\]
Then
\begin{equation}\tag{$\bar{F}^+(k_1,k_2,j)$}\label{Fkalpha+_bar}
\sum_{i \in [d]}\alpha^{k_1,k_2}_i(j)
\left|x_i-t^0_i-\sum_{k \in N} (t^k_i-t^0_i)z_k\right|
+\sum_{k \in N}\mu^{k_1,k_2}_k(j) z_k
\leq M^{k_1,k_2}_0(j).
\end{equation}

Observe that when we substitute in $z = \mathbf{e}_k$\,, \ref{Fkalpha} becomes a facet-describing inequality for $\Po_k$\,. So we start with \ref{Fkalpha} with $z = \mathbf{e}_k$ as the starting inequality. We check that the inequality \ref{Fkalpha+_bar} is valid for $\D$ by checking it is valid for extreme points of $\bar{\Po}_k$ for all $k \in N_0$\,. When $z = \mathbf{e}_0$\,, the left hand side of the inequality \ref{Fkalpha+_bar} becomes: $\sum_{i \in [d]}\alpha^{k_1,k_2}_i(j)
\left|x_i-t^0_i\right| \leq M^{k_1,k_2}_0(j)$ by definition. When $z = \mathbf{e}_l$ for $l \in N$\,, the inequality \ref{Fkalpha+_bar} becomes: $\sum_{i \in [d]}\alpha^{k_1,k_2}_i(j)
\left|x_i-t^l_i\right|
+ M^{k_1,k_2}_0(j)
-
M^{k_1,k_2}_l(j) \leq M^{k_1,k_2}_l(j) + M^{k_1,k_2}_0(j)
-
M^{k_1,k_2}_l(j) = M^{k_1,k_2}_0(j)$\,. Therefore, \ref{Fkalpha+_bar} is valid for $\D$.

\begin{itemize}
    \item \(\mathcal{F}^+\) has all inequalities from $\mathcal{F}$, but with improved lifting coefficients, as above.
    \item \(\mathcal{L}^+\) has all inequalities from $\mathcal{L}$, but with improved lifting coefficients, as above.
\end{itemize}
Comparing $\mathcal{F}^+$ and $\mathcal{L}^+$,
with $\mathcal{F}$ and $\mathcal{L}$,
we can learn something about whether, in the presence of some side constraints,
we can get significant improvements 
by solving auxiliary optimization problems
to improve upon our closed-form lifting coefficients. 

\item From \(\mathcal{F}^+,\mathcal{L}^+\)
to \(\mathcal{F}^{++},\mathcal{L}^{++}\)

Consider the vertex-cut constraints: $\left|x_{i_k}-t^k_{i_k}\right|
\leq
\frac{1}{2 \bm{a}^k_{i_k}}$, which were used in the definition of $\Po_k^+$\,.

 For \(\mathcal{F}^{++},\mathcal{L}^{++}\),
on top of \(\mathcal{F}^+\) and \(\mathcal{L}^+\), we add the lifted ``vertex-cut constraints'', defined as follows. Let
\[
\tau_k(x):=|x_{i_k}-t_{i_k}^k|,
\qquad
\beta_k:=\frac{1}{2 \bm{a}_{i_k}^k}.
\]
Then, we have the \emph{lifted vertex-cut constraints}:
\begin{align*}
k=0:\quad
&\tau_0(x)+\sum_{i\in N}
\left(\beta_0-\max_{x\in \Po_i^+}\tau_0(x)\right)z_i
\le \beta_0,
\\
k\in N:\quad
&\tau_k(x)
+\sum_{l\in N\setminus\{k\}}
\left(\max_{x\in \Po_0^+}\tau_k(x)-\max_{x\in \Po_l^+}\tau_k(x)\right)z_l
\\
&
+\left(\max_{x\in \Po_0^+}\tau_k(x)-\beta_k\right)z_k
\le
\max_{x\in \Po_0^+}\tau_k(x).
\end{align*}
Comparing $\mathcal{F}^{++}$ and $\mathcal{L}^{++}$,
with $\mathcal{F}^+$ and $\mathcal{L}^+$,
we can learn something about whether, in the presence of some side constraints on individual polytopes,
we can get significant improvements 
by solving auxiliary optimization problems
to lift the side-constraint coefficients, obtaining inequalities that are valid
for all of the $\Po_k$\,.
\end{itemize}

%%%%%%%%%%%%%%%%%%%%%%%

\subsection{Comparisons}
We compare relaxation methods against the {\it exact $\mathcal{O}$ptimal solution}:
\begin{equation}\tag{$\mathcal{O}$}
\mathfrak{z}(\mathcal{O})
:=
\max_{k\in N_0}
\left\{
g_k + c^\top x : x \in \Po_k^+
\right\},
\end{equation}

 For each \((d,n)\), with
 $3\leq d\leq 12$, $2\leq n \leq 20$,
 we compute the optimal objective value $\mathfrak{z}(\cdot)$ of each method
\[
X \in \{\mathcal{F}^{++},\,\mathcal{L}^{++},\,\mathcal{F}^{+},\,\mathcal{L}^{+},\,\mathcal{F},\,\mathcal{L}\},
\]
and compare it with the exact optimal value \(\mathfrak{z}(\mathcal{O})\).

We evaluate the value of method \(X\) compared to
our weakest relaxation $\mathcal{L}$ by the fraction of the gap closed:
\[
1-\frac{\mathfrak{z}(X)-\mathfrak{z}(\mathcal{O})}{\mathfrak{z}(\mathcal{L})-\mathfrak{z}(\mathcal{O})}.
\]
Therefore, if the value is near $1$, the gap left by     $\mathcal{L}$ is almost completely closed by $X$;
if the value is near $0$,  the gap left by 
    $\mathcal{L}$ is barely improved by $X$.
In Figure \ref{fig:heat}, for each method \(X\), we plot a
heat map resulting from our experiments.
For each heat map:
\begin{figure}%[t!]
\includegraphics[
    width=\textwidth,
    height=0.95\textheight,
    keepaspectratio
]{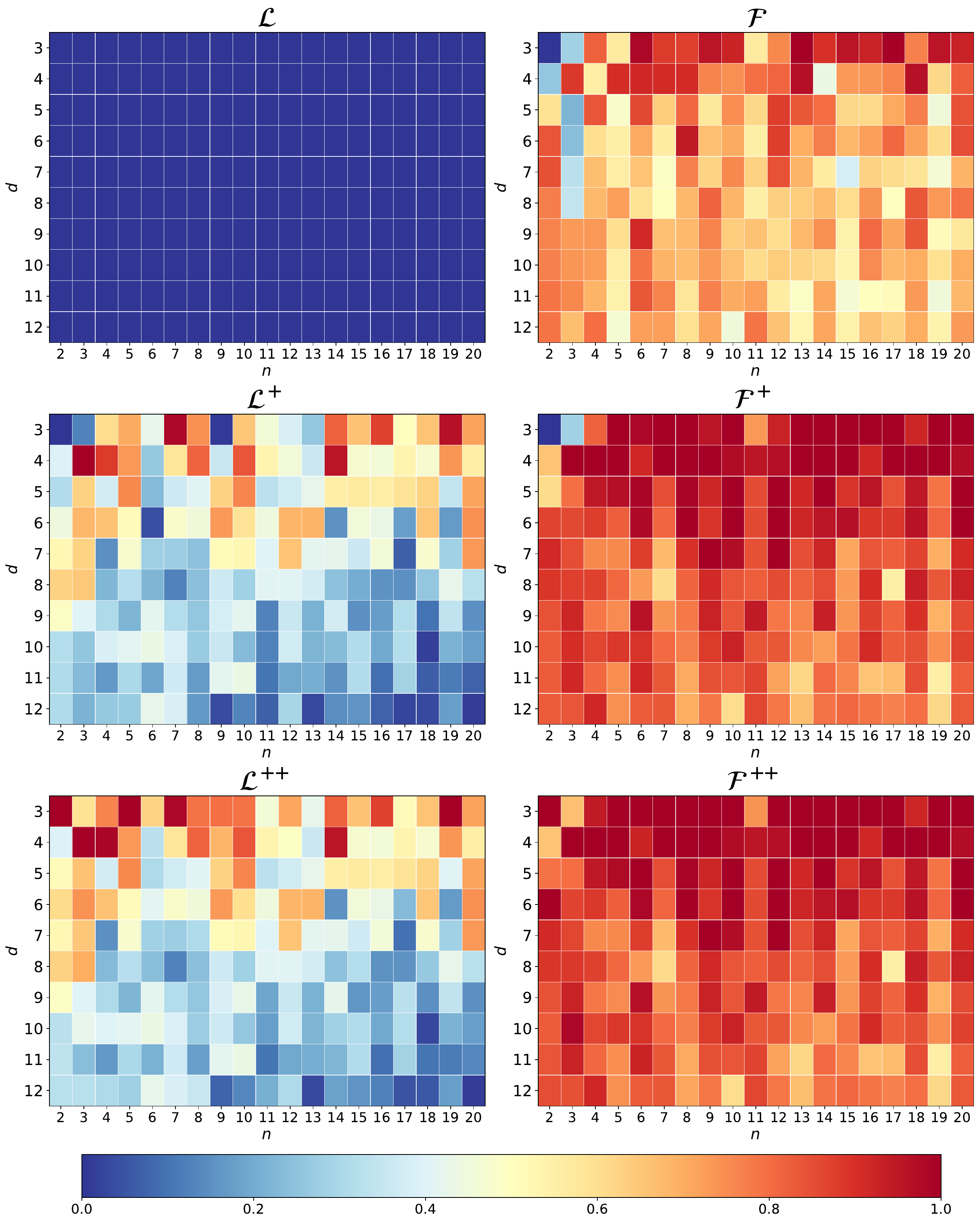}
\caption{Heat maps for gap closing}\label{fig:heat}
\end{figure}

\FloatBarrier

\begin{itemize}
    \item the rows are indexed by \(d\).
    \item the columns are indexed by \(n\).
    \item  color is determined by the value (i.e.,
    fraction of the gap closed), with more red indicating close to $1$, and more blue indicating close to $0$.
\end{itemize}

We make several observations:
\begin{itemize}
    \item Non-lifting inequalities close quite a lot of the gap left by lifting inequalities, whether or not we enhance with the + and ++ techniques.
    \item Already the + technique closes a lot of gap, with only a modest improvement with ++. 
    Still, there are instances, for example $d=2$, $n=3$, where the ++ technique is critical. 
    \item There is a general trend of increasing difficulty in closing gap, as $d$ and $n$ increase. But the deterioration is much less pronounced, if we start with all facets of $\D$ (i.e., if we do not exclude non-lifting vertical facets of $\D$). 
\end{itemize}

Of course, with more side constraints,
we could have more difficulty in closing gaps, but we believe that we have some good evidence that the techniques that we presented have some robustness for general use. 

%%%%%%%%%%%%%%%%%%%%

\section{Outlook}\label{sec:outlook}

With regard to the theory that we developed in \S\ref{sec:newresults_cross},
We are:
(i)
looking at generalizing the results 
generalized cross polytopes
 to ``incomplete generalized cross polytopes'' (which would cover the case
in which the input polytopes are arbitrary axis-aligned $d$-simplices, e.g., the example from \cite[Proposition 9]{QL_DAM2025}),
(ii) looking at  determining if the 
non-lifting vertical facets of $\D$ are MIR inequalities (see \cite{cornrio}), motivated by \cite[Section 5]{QL_DAM2025},
(iii)
investigating a generalization of Example \ref{ex:tree} to a broad family of facets that would include those described by \ref{Fkalpha},
and
(iv) applying and further developing our work in the context of 
solving mixed-integer linear optimization models that contain piecewise-linear functions of
a small number of variables (see \cite{LeeWilson}, for example).

%%%%%%%%%%%%%%%%

\appendix

\section{Proof of Theorem \ref{thm:facets_general_n}}\label{app:facet}
The proof follows the ideas of Theorem \ref{thm:facet}, 
but with heavier notation.
\begin{proof}
First, we demonstrate that \ref{Fkalpha} is valid for $\mathcal{D}$.
The extreme points of $\bar{\Po}_k$ are $\hat{x}^{k,i} := \left\{\left( \genfrac{}{}{0pt}{1}{t_i \frac{1}{\bm{a}_i^k} \mathbf{e}_i}{\mathbf{e_k}} \right) \right\}$, $i \in [d]$, where $t_i\in\{\pm1\}$, for $k \in N_0$.
When $z = \mathbf{e}_0$\,, we look at extreme points of $\bar{\Po}_0$ and plug into \ref{Fkalpha}:
\begin{align*}
\sum_{l=1}^d s_l \alpha^{k_1,k_2}_l(j) \hat{x}^{0,i}_l + \sum_{k=1}^n (M^{k_1,k_2}_0(j) - M^{k_1,k_2}_k(j) ) \cdot 0 = s_i t_i \frac{\alpha^{k_1,k_2}_i(j)}{\bm{a}^0_i} \leq \frac{\alpha^{k_1,k_2}_i(j)}{\bm{a}^0_i} \leq M^{k_1,k_2}_0(j).
\end{align*}
So, \ref{Fkalpha} holds for all extreme points of $\bar{\Po}_0$\,, and therefore 
holds for all of $\bar{\Po}_0$\,. 

Similarly, when $z = \mathbf{e}_k$ for $k \in N$, we look at extreme points of $\bar{\Po}_k$ and plug into \ref{Fkalpha}:
\begin{align*}
& \sum_{l=1}^d s_l \alpha^{k_1,k_2}_l(j) \hat{x}^{k,i}_l +  (M^{k_1,k_2}_0(j) - M^{k_1,k_2}_k(j) ) \cdot 1 \\
& = s_i t_i \frac{\alpha^{k_1,k_2}_i(j)}{\bm{a}^k_i} +  (M^{k_1,k_2}_0(j) - M^{k_1,k_2}_k(j) ) \\
& \leq \left( \frac{\alpha^{k_1,k_2}_i(j)}{\bm{a}^k_i} - M^{k_1,k_2}_k(j) \right) +  M^{k_1,k_2}_0(j)  \leq M^{k_1,k_2}_0(j).
\end{align*}
So \ref{Fkalpha} holds for all extreme points of $\bar{\Po}_k$\,, $k \in N$, and therefore holds for all of $\bar{\Po}_k$\,, $k \in N$. 

In summary, the inequalities \ref{Fkalpha} are valid for $\D$. We now find out $d+n$ affinely-independent points in $\D$ satisfying \ref{Fkalpha} as equations.

If $\hat{k} = 0$, then $z = \mathbf{e}_0$\,, so the inequality reduced to $\sum_{l=1}^d s_l \alpha^{k_1,k_2}_l(j) \hat{x}^{0,i}_l \leq M^{k_1,k_2}_0(j)$. If $\hat{k} \in N$, then $z = \mathbf{e}_{\hat{k}}$\,, so $\sum_{l=1}^d s_l \alpha^{k_1,k_2}_l(j) \hat{x}^{\hat{k},i}_l +  (M^{k_1,k_2}_0(j) - M^{k_1,k_2}_{\hat{k}}(j) ) \cdot 1  = \left( \sum_{l=1}^d s_l \alpha^{k_1,k_2}_l(j) \hat{x}^{\hat{k},i}_l - M^{k_1,k_2}_{\hat{k}}(j) \right) +  M^{k_1,k_2}_0(j) \leq M^{k_1,k_2}_0(j) \Rightarrow \sum_{l=1}^d s_l \alpha^{k_1,k_2}_l(j) \hat{x}^{\hat{k},i}_l \leq M^{k_1,k_2}_{\hat{k}}(j) $.

For $\hat{k}\in N_0$\,, let $\hat{x}^{\hat{k},i} = \left( \genfrac{}{}{0pt}{1}{s_i \frac{1}{\bm{a}_i^{\hat{k}}} \mathbf{e}_i}{\mathbf{e}_{\hat{k}}} \right)$. Then $\hat{x}^{\hat{k},i}$ satisfy \ref{Fkalpha} as equations if and only if $\sum_{l=1}^d s_l \alpha^{k_1,k_2}_l(j) \hat{x}^{\hat{k},i}_l = \frac{\alpha^{k_1,k_2}_i(j)}{\bm{a}^{\hat{k}}_i} = M^{k_1,k_2}_{\hat{k}}(j) = \frac{\alpha^{k_1,k_2}_i(j)}{\bm{a}^{\hat{k}}_i} = \max_{ i \in [d]}\frac{\alpha^{k_1,k_2}_i(j)}{\bm{a}^k_i}$, that is, if and only if $i$ attains the maximum in the definition of $M^{k_1,k_2}_{\hat{k}}(j)$. 

Consider the extreme points $\hat{x}^{k_1,i} := \left( \genfrac{}{}{0pt}{1}{s_i \frac{1}{\bm{a}_i^{k_1}} \mathbf{e}_i}{\mathbf{e}_{k_1}} \right) $ for $\bar{\Po}_{k_1}$\,. For $i \in \mathcal{L}^{k_1,k_2}(j) \cup \mathcal{E}^{k_1,k_2}(j)$ , we have $\frac{ \alpha_i^{k_1,k_2}(j)}{\bm{a}^{k_1}_i} = \frac{ \bm{a}_i^{k_1}}{\bm{a}^{k_1}_i} = 1$. For $i \in \mathcal{G}^{k_1,k_2}(j)$, we have $\frac{ \alpha_i^{k_1,k_2}(j)}{\bm{a}^{k_1}_i} = \frac{r^{k_1,k_2}_j a_i^{k_2}}{\bm{a}^{k_1}_i} = \frac{r^{k_1,k_2}_j}{r^{k_1,k_2}_i} < 1$. So $\max_{ i \in [d]}\frac{\alpha^{k_1,k_2}_i(j)}{\bm{a}^{k_1}_i} = 1$ and choose $i^{k_1} \in \argmax_{ i \in [d]}\frac{\alpha^{k_1,k_2}_i(j)}{\bm{a}^{k_1}_i} $ for every $i^{k_1} \in \mathcal{L}^{k_1,k_2}(j) \cup \mathcal{E}^{k_1,k_2}(j)$. Therefore, $\hat{x}^{k_1,i^{k_1}}$ satisfy \ref{Fkalpha} as equations, and we choose $|\mathcal{L}^{k_1,k_2}(j)| + |\mathcal{E}^{k_1,k_2}(j)|$ points here. 

Consider the extreme points $\hat{x}^{k_2,i}:= \left( \genfrac{}{}{0pt}{1}{s_i \frac{1}{\bm{a}_i^{k_2}} \mathbf{e}_i}{\mathbf{e}_{k_2}} \right) $ for $\bar{\Po}_{k_2}$\,. For $i \in \mathcal{L}^{k_1,k_2}(j)$, we have $\frac{ \alpha_i^{k_1,k_2}(j)}{\bm{a}^{k_2}_i} = \frac{ \bm{a}_i^{k_1}}{\bm{a}^{k_2}_i} = r^{k_1,k_2}_i < r^{k_1,k_2}_j$.  For $i \in \mathcal{E}^{k_1,k_2}(j) \cup \mathcal{G}^{k_1,k_2}(j)$, $\frac{ \alpha_i^{k_1,k_2}(j)}{\bm{a}^{k_2}_i}  = \frac{r^{k_1,k_2}_j \bm{a}_i^{k_2}}{\bm{a}^{k_2}_i} = r^{k_1,k_2}_j$. So $\max_{ i \in [d]}\frac{\alpha^{k_1,k_2}_i(j)}{\bm{a}^{k_2}_i} = r^{k_1,k_2}_j$. Choose 
$i^{k_2} \in \argmax_{ i \in [d]}\frac{\alpha^{k_1,k_2}_i(j)}{\bm{a}^{k_2}_i} $ for every $i^{k_2} \in \mathcal{E}^{k_1,k_2}(j) \cup \mathcal{G}^{k_1,k_2}(j)$. Therefore, $\hat{x}^{{k}_2,i^{k_2}}$ satisfy \ref{Fkalpha} as equations, and we can choose $|\mathcal{G}^{k_1,k_2}(j)|$ points here for $i^{k_2} \in \mathcal{G}^{k_1,k_2}(j)$. Let $\hat{\imath} \in \mathcal{E}^{k_1,k_2}(j)$. We pick $\left( \genfrac{}{}{0pt}{1}{s_{\hat{\imath}} \frac{1}{\bm{a}_{\hat{\imath}}^{k_2}} \mathbf{e}_{\hat{\imath}}}{\mathbf{e}_{k_2}} \right)$, and this point also satisfy \ref{Fkalpha} as an equation. Because $|\mathcal{L}^{k_1,k_2}(j)|+|\mathcal{G}^{k_1,k_2}(j)|+|\mathcal{E}^{k_1,k_2}(j)| = d$ and we choose one more point from $\mathcal{E}^{k_1,k_2}(j)$, we have $d+1$ points here. 

Then, for each $k \in N_0 \setminus \{k_1,k_2\}$, we just choose $i^k \in \argmax_{ i \in [d]}\frac{\alpha^{k_1,k_2}_i(j)}{\bm{a}^{k}_i}$. Let  $\hat{x}^{k,i^k}:= \left( \genfrac{}{}{0pt}{1}{s_{i^k} \frac{1}{\bm{a}_{i^k}^{k}} \mathbf{e}_{i^k}}{\mathbf{e}_{k}} \right)$ be extreme points of $\bar{\Po}_{k}$\,. By construction, $\frac{\alpha^{k_1,k_2}_{i^k}(j)}{\bm{a}^k_{i^k}} = M^{k_1,k_2}_k(j)$\,, so each point $\hat{x}^{k,i^k}$ also satisfies \ref{Fkalpha} as equations. Then, we have $n+1-2=n-1$ points here. So totally we have $d+n$ points satisfying \ref{Fkalpha} as equations.

It remains to show that those $d+n$ points are affinely independent. 
We arrange these points as columns of a matrix. After appropriately 
permuting rows and columns, we arrive at the matrix
 \[
S :=
\begin{array}{c@{\hspace{0.8em}}c}
&
\begin{array}{@{}ccccc@{}}
\B{|\mathcal L^{k_1,k_2}(j)|} &
\B{|\mathcal G^{k_1,k_2}(j)|} &
\B{|\mathcal E^{k_1,k_2}(j)|} &
\makebox[4.8em][c]{\makebox[0pt][c]{\hspace{-0.5em}\ensuremath{1}}} &
\B[9.0em]{n-1}
\end{array}
\\[5pt]
\begin{array}{c}
|\mathcal L^{k_1,k_2}(j)| \\[0.28em]
|\mathcal G^{k_1,k_2}(j)| \\[0.5em]
|\mathcal E^{k_1,k_2}(j)| \\[0.5em]
n
\end{array}
&
\left[
\begin{array}{@{}ccc|cc@{}}
\B{D_{\mathcal L}} &
\B{\mathbf{0}} &
\B{\mathbf{0}} &
 &  \\[0.25em]
\B{\mathbf{0}} &
\B{D_{\mathcal G}} &
\B{\mathbf{0}} &
\B[4.8em]{\frac{s_{\hat{\imath}}}{\bm{a}_{\hat{\imath}}^{k_2}} \mathbf e_{\hat{\imath}}} & [s_{i^k} \frac{1}{\bm{a}_{i^k}^{k}} \mathbf{e}_{i^k}]_{k \in N_0 \setminus \{k_1,k_2\}}\\[0.25em]
\B{\mathbf{0}} &
\B{\mathbf{0}} &
\B{D_{\mathcal E}} &
\B[4.8em] & \\[0.25em]
\hline
\B{\mathbf{e}_{k_1} \mathbf{e}^\top} &
\B{\mathbf{e}_{k_2} \mathbf{e}^\top} &
\B{\mathbf{e}_{k_1} \mathbf{e}^\top} &
\B[4.8em]{\mathbf{e}_{k_2}} & \B{[\mathbf{e}_{k}]_{k \in N_0 \setminus \{k_1,k_2\}}}
\end{array}
\right],
\end{array}
\]
where $D_{\mathcal{L}} := {\rm Diag}\left( \frac{s_{i}}{\bm{a}^{k_1}_{i}}: i \in \mathcal{L}^{k_1,k_2}(j) \right)$, $D_{\mathcal{G}} := {\rm Diag}\left( \frac{s_{i}}{\bm{a}^{k_2}_{i}}: i \in \mathcal{G}^{k_1,k_2}(j)\right)$, $D_{\mathcal{E}} := {\rm Diag}\left( \frac{s_{i}}{\bm{a}^{k_1}_{i}}: i \in \mathcal{E}^{k_1,k_2}(j)\right)$. 

The matrix $S$ has the form  
$     \left[
\begin{array}{cc}
         A &  B\\
         F & G
\end{array}
    \right]$.
Because $D_{\mathcal{L}}\,, D_{\mathcal{E}}\,, D_{\mathcal{G}}$ are diagonal with nonzero diagonal entries, $A$ is nonsingular. Now, we use first $d$ columns to eliminate the $x$ coordinates (first $d$ rows) of the last $n$ columns. After the column operations, the matrix $S$ transforms to  
 $   S' :=
     \left[
\begin{array}{cc}
         A &  \mathbf{0} \\
         F & H
\end{array}
    \right]$,
where $H = G - F A^{-1} B$.
Then, $\det(S) = \det(S') = \det(A) \det(H)$. It remains to show that $H$ is nonsingular. 
The first column of $H$ after elimination is $h_* = \mathbf{e}_{k_2} - \frac{a^{k_1}_{\hat{\imath}}}{a^{k_2}_{\hat{\imath}}} \mathbf{e}_{k_1}$\,, where $ \frac{a^{k_1}_{\hat{\imath}}}{a^{k_2}_{\hat{\imath}}} \neq 0$.
For each $k \in N_0 \setminus \{k_1,k_2\}$,  the corresponding column of $H$ is $h_k = \mathbf{e}_k - \lambda_k \mathbf{e}_{\tau({k})}$\,, where
$\lambda_k := \frac{a^{\tau({k})}_{i^k}}{a^{k}_{i^k}} \neq 0$, and 
\[
\tau({k}):= 
\begin{cases}
k_1 \,, & \mbox{for } i^k \in \mathcal{L}^{k_1,k_2}(j) \cup \mathcal{E}^{k_1,k_2}(j); \\ 
k_2\,, & \mbox{for } i^k \in \mathcal{G}^{k_1,k_2}(j).
\end{cases}
\]

We have two cases:

\noindent \underline{Case $0<k_1<k_2$}\,:

Because $k \in N_0 \setminus \{k_1, k_2\}$, $k=0$ appears in the matrix $H$. In particular, $h_0 =  -\lambda_0 \mathbf{e}_{\tau({0})}$\,. For every $k \in N \setminus \{k_1, k_2\}$, the row $k$ of $H$ has exactly one nonzero entry, which is the entry $1$ in the column $h_k$\,. Laplacian expansion along those rows with entry $1$ reduces the determinant to the $2 \times 2$ matrix in rows $k_1\,,k_2$ and columns $h_*\,, h_0$\,.

If $\tau({0}) = k_1$\,, the remaining matrix is
$\begin{bmatrix}
         - \frac{a^{k_1}_{\hat{\imath}}}{a^{k_2}_{\hat{\imath}}} &  -\lambda_0 \\
         1 & 0
\end{bmatrix}$, which has determinant $\lambda_0 \neq 0$.

\vspace{5pt}

If $\tau({0}) = k_2$\,, the remaining matrix is
$\begin{bmatrix}
         - \frac{a^{k_1}_{\hat{\imath}}}{a^{k_2}_{\hat{\imath}}} &  0 \\
         1 & -\lambda_0
\end{bmatrix}$, which has determinant $\frac{a^{k_1}_{\hat{\imath}}}{a^{k_2}_{\hat{\imath}}} \lambda_0 \neq 0$.

Therefore, $H$ is nonsingular in this case.

\noindent \underline{Case $0=k_1<k_2$}\,:

Then $\mathbf{e}_{k_1} = \mathbf{0}$, and $h_* = \mathbf{e}_{k_2}$\,. For every $k \in N_0 \setminus \{0, k_2\}$, we have $h_k = \mathbf{e}_k - \lambda_k \mathbf{e}_{\tau({k})}$\,. Then, if $\tau(k) = k_1 = 0$, then $h_k = \mathbf{e}_k$\,. If $\tau(k) = k_2$\,, we have $h_k = \mathbf{e}_k - \lambda_k \mathbf{e}_{k_{2}}$\,. In either case, the row $k$ has an entry $1$ in column $h_k$\,. Expanding along all rows $k \in N_0 \setminus \{0, k_2\}$, we are left with row $k_2$ and column $h_* = \mathbf{e}_{k_2}$\,, which entry is $1$. Therefore, $H$ is nonsingular in this case.

Thus, in both cases, $H$ is nonsingular. Then $A$ is nonsingular, and therefore $S$ is nonsingular. Hence those $d+n$ points are linearly independent, and therefore affinely independent.
\qed
\end{proof}

%%%%%%%%%%%%%%%%%%%%%%%%%%%%%%%%%%%%%%%%%%%%%%%%%%%%%%%%%%%%%%%%%%%

\section{Proof of Theorem \ref{thm:R_generaln} }\label{app:lift}
The proof follows the ideas of the proof of Theorem  \ref{thm:ratiosettheorem},
but with heavier notation.

\begin{proof}
Let $r^{k_1,k_2}_{\min} := \min\{r^{k_1,k_2}_i: i \in [d]\}$, $r^{k_1,k_2}_{\max} := \max\{r^{k_1,k_2}_i: i \in [d]\}$. Also, let $S^{k_1,k_2}_L(r_{{j}}):= \{ i \in [d]: r^{k_1,k_2}_i \leq r^{k_1,k_2}_{{j}}\}$, $S^{k_1,k_2}_G(r_{{j}}):= \{ i \in [d]: r^{k_1,k_2}_i \geq r^{k_1,k_2}_{{j}} \}$.

$(\Rightarrow):$ Suppose that every inequality \ref{Fkalpha}, for a fixed pair $k_1,k_2$, is a lifting inequality associated with $\Po_{k_1}$, or $\Po_{k_2}$
Prove by contradiction. Suppose $|R^{k_1,k_2}|>2$ for some $(k_1,k_2) \in N_0 \times N_0$. Let $r^{k_1,k_2}_{\hat{\jmath}} \in R^{k_1,k_2}$ such that $r^{k_1,k_2}_{\min} < r^{k_1,k_2}_{\hat{\jmath}} < r^{k_1,k_2}_{\max}$. Set $S^{k_1,k_2}_L(r_{\hat{\jmath}}):= \{ i \in [d]: r^{k_1,k_2}_i \leq r^{k_1,k_2}_{\hat{\jmath}}\}$, $S^{k_1,k_2}_G(r^{k_1,k_2}_{\hat{\jmath}}):= \{ i \in [d]: r^{k_1,k_2}_i \geq r^{k_1,k_2}_{\hat{\jmath}} \}$. Then, $S^{k_1,k_2}_L(r_{\hat{\jmath}}), S^{k_1,k_2}_G(r_{\hat{\jmath}})$ are proper subsets of $[d]$, and therefore every facet uses at least two points with indices from each set, and the facet is not a lifting inequality.

$(\Leftarrow):$ Suppose $|R^{k_1,k_2}| = 1$, say $r^{k_1,k_2}_i = r^{k_1,k_2}$ for all $i$. Then, $S^{k_1,k_2}_L(r^{k_1,k_2}_i) = S^{k_1,k_2}_G(r^{k_1,k_2}_i) = [d], \forall i \in [d]$. We can choose $d$ affinely-independent points from $\Po_{k_1}$ and one extreme point from $\Po_{k_2}$\,. Similarly, we can choose $d$ affinely-independent points from $\Po_{k_2}$ and one extreme point from $\Po_{k_1}$\,. It is necessary to choose also one point from each of the remaining $n-1$ polytopes to have $d+n$ affinely-independent points spanning a \emph{vertical} facet.

Suppose that $|R^{k_1,k_2}| = 2$, say $R^{k_1,k_2} = \{r^{k_1,k_2}_{\min}, r^{k_1,k_2}_{\max}\}$ with $r^{k_1,k_2}_{\min} < r^{k_1,k_2}_{\max}$. Then, we have two cases:
\begin{enumerate}
    \item $S^{k_1,k_2}_L(r^{k_1,k_2}_{\min}) = \{i: r^{k_1,k_2}_i \leq r^{k_1,k_2}_{\min} \} = \{i: r^{k_1,k_2}_i = r^{k_1,k_2}_{\min} \} \neq \varnothing$, and $S^{k_1,k_2}_G(r^{k_1,k_2}_{\min}) = \{i: r^{k_1,k_2}_i \geq r^{k_1,k_2}_{\min} \} = [d]$. Then, we choose $d$ affinely-independent points $\left( \genfrac{}{}{0pt}{1}{s_i \frac{1}{\bm{a}_i^{k_2}} \mathbf{e}_i}{\mathbf{e}_{k_2}}\right)$ for $i \in S^{k_1,k_2}_G(r^{k_1,k_2}_{\min})$, and one more point $\left( \genfrac{}{}{0pt}{1}{s_i \frac{1}{\bm{a}_i^{k_1}} \mathbf{e}_i}{\mathbf{e}_{k_1}}\right)$ for $i \in S^{k_1,k_2}_L(r^{k_1,k_2}_{\min})$. Then, we choose one point from each of the remaining $n-1$ polytopes.
    So the inequality 
    \ref{Fkalpha}
   is a full optimal big-M lifting inequality starting from $\Po_{k_2}$\,. 
    \item $S^{k_1,k_2}_L(r^{k_1,k_2}_{\max}) = \{i: r^{k_1,k_2}_i \leq r^{k_1,k_2}_{\max} \} = [d]$, and $S^{k_1,k_2}_G(r^{k_1,k_2}_{\max}) = \{i: r^{k_1,k_2}_i \geq r^{k_1,k_2}_{\max} \}  = \{i: r^{k_1,k_2}_i = r^{k_1,k_2}_{\max} \} = \varnothing$. Then, we choose $d$ affinely-independent points $\left( \genfrac{}{}{0pt}{1}{s_i \frac{1}{\bm{a}_i^{k_1}} \mathbf{e}_i}{\mathbf{e}_{k_1}} \right)$ for $i \in S^{k_1,k_2}_L(r^{k_1,k_2}_{\max})$, and one more point $\left( \genfrac{}{}{0pt}{1}{s_i \frac{1}{\bm{a}_i^{k_2}} \mathbf{e}_i}{\mathbf{e}_{k_2}}\right)$ for $i \in S^{k_1,k_2}_G(r^{k_1,k_2}_{\max})$. Then, we choose one point from each of the remaining $n-1$ polytopes.
    So the inequality \ref{Fkalpha} is a full optimal big-M lifting inequalities starting from $\Po_{k_1}$. \qed
\end{enumerate}  
\end{proof}

%%%%%%%%%%%%%%%%%%%%%%%%%%

\section{Details for Example \ref{ex:tree}}\label{app:ex}

\noindent \underline{Case $k_1 = 0$, $k_2=1$}:

\noindent \underline{$j=1$}: $r^{0,1}_1 = \frac{1}{10}$, $r^{0,1}_1 \bm{a}^1 = \frac{1}{10} \left(10, \frac{1}{2}, \frac{1}{2}, 10 \right) = \left(1, \frac{1}{20}, \frac{1}{20}, 1 \right)$, $\alpha^{0,1}(1) = \left(1, \frac{1}{20}, \frac{1}{20}, 1 \right)$, $M^{0,1}_0(1) = \max_{i \in [4]} \left\{ \frac{\alpha^{0,1}_i(1)}{\bm{a}^0_i} \right\} = \max \left\{ \frac{1}{1}, \frac{1/20}{1},\frac{1/20}{10},\frac{1}{10} \right\} = 1$, $M^{0,1}_1(1) = \max_{i \in [4]} \left\{ \frac{\alpha^{0,1}_i(1)}{\bm{a}^1_i} \right\} = \max \left\{ \frac{1}{10}, \frac{1/20}{1/2},\frac{1/20}{1/2},\frac{1}{10} \right\} = \frac{1}{10}$, $M^{0,1}_2(1) = \max_{i \in [4]} \left\{ \frac{\alpha^{0,1}_i(1)}{\bm{a}^2_i} \right\} = \max \left\{ \frac{1}{10}, \frac{1/20}{10},\frac{1/20}{1/3},\frac{1}{1/3} \right\} = 3$, which gives $x_1 + \frac{1}{20}x_2 + \frac{1}{20}x_3 + x_4 + \frac{9}{10}z_1 - 2 z_2 \leq 1$.

\smallskip

\noindent \underline{$j=2$}: $r^{0,1}_2 = 2$, $r^{0,1}_2 \bm{a}^1 = 2 \left(10, \frac{1}{2}, \frac{1}{2}, 10 \right) = (20, 1, 1, 20)$, $\alpha^{0,1}(2) = (1, 1, 1, 10)$, $M^{0,1}_0(2) = \max_{i \in [4]} \left\{\frac{\alpha^{0,1}_i(2)}{\bm{a}^0_i} \right\} = \max \left\{\frac{1}{1}, \frac{1}{1},\frac{1}{10},\frac{10}{10} \right\} = 1$, $M^{0,1}_1(2) = \max_{i \in [4]} \left\{\frac{\alpha^{0,1}_i(2)}{\bm{a}^1_i} \right\} = \max \left\{\frac{1}{10}, \frac{1}{1/2},\frac{1}{1/2},\frac{10}{10} \right\} = 2$, $M^{0,1}_2(2) = \max_{i \in [4]} \left\{\frac{\alpha^{0,1}_i(2)}{\bm{a}^2_i} \right\} = \max \left\{\frac{1}{10}, \frac{1}{10},\frac{1}{1/3},\frac{10}{1/3} \right\} = 30$, which gives $x_1 + x_2 + x_3 + 10 x_4 - z_1 - 29 z_2 \leq 1$.

\smallskip

\noindent \underline{$j=3$}: $r^{0,1}_3 = 20$, $r^{0,1}_3 \bm{a}^1 = 20 \left(10, \frac{1}{2}, \frac{1}{2}, 10 \right) = (200, 10, 10, 200)$, $\alpha^{0,1}(3) = (1, 1, 10, 10)$, $M^{0,1}_0(3) = \max_{i \in [4]} \left\{\frac{\alpha^{0,1}_i(3)}{\bm{a}^0_i} \right\} = \max \left\{\frac{1}{1}, \frac{1}{1},\frac{10}{10},\frac{10}{10} \right\} = 1$, $M^{0,1}_1(3) = \max_{i \in [4]} \left\{\frac{\alpha^{0,1}_i(3)}{\bm{a}^1_i} \right\} = \max \left\{\frac{1}{10}, \frac{1}{1/2},\frac{10}{1/2},\frac{10}{10} \right\} = 20$, $M^{0,1}_2(3) = \max_{i \in [4]} \left\{\frac{\alpha^{0,1}_i(3)}{\bm{a}^2_i} \right\} = \max \left\{\frac{1}{10}, \frac{1}{10},\frac{10}{1/3},\frac{10}{1/3} \right\} = 30$, which gives $x_1 + x_2 + 10 x_3 + 10 x_4 - 19 z_1 - 29 z_2 \leq 1$.

\smallskip

\noindent \underline{$j=4$}: $r^{0,1}_4 = 1$, $r^{0,1}_4 \bm{a}^1 = 1 \left(10, \frac{1}{2}, \frac{1}{2}, 10 \right) = \left(10, \frac{1}{2}, \frac{1}{2}, 10 \right)$, $\alpha^{0,1}(4) = \left(1, \frac{1}{2}, \frac{1}{2}, 10 \right)$, $M^{0,1}_0(4) = \max_{i \in [4]} \left\{\frac{\alpha^{0,1}_i(4)}{\bm{a}^0_i} \right\} = \max \left\{\frac{1}{1}, \frac{1/2}{1},\frac{1/2}{10},\frac{10}{10} \right\} = 1$, $M^{0,1}_1(4) = \max_{i \in [4]} \left\{\frac{\alpha^{0,1}_i(4)}{\bm{a}^1_i} \right\} = \max \left\{\frac{1}{10}, \frac{1/2}{1/2},\frac{1/2}{1/2},\frac{10}{10} \right\} = 1$, $M^{0,1}_2(4) = \max_{i \in [4]} \left\{\frac{\alpha^{0,1}_i(4)}{\bm{a}^2_i} \right\} = \max \left\{\frac{1}{10}, \frac{1/2}{10},\frac{1/2}{1/3},\frac{10}{1/3} \right\}= 30$, which gives $x_1 + \frac{1}{2}x_2 + \frac{1}{2}x_3 + 10x_4 - 29 z_2 \leq 1$.

\bigskip

\noindent \underline{Case $k_1 = 0$, $k_2=2$}:

\noindent \underline{$j=1$}: $r^{0,2}_1 = \frac{1}{10}$, $r^{0,2}_1 \bm{a}^2 = \frac{1}{10}\left(10, 10, \frac{1}{3}, \frac{1}{3} \right) = \left(1, 1, \frac{1}{30}, \frac{1}{30} \right)$, $\alpha^{0,2}(1) = \left(1, 1, \frac{1}{30}, \frac{1}{30} \right)$, $M^{0,2}_0(1) = \max_{i \in [4]} \left\{ \frac{\alpha^{0,2}_i(1)}{\bm{a}^0_i} \right\} = \max \left\{ \frac{1}{1}, \frac{1}{1},\frac{1/30}{10},\frac{1/30}{10}\right \} = 1$, $M^{0,2}_1(1) = \max_{i \in [4]} \left\{\frac{\alpha^{0,2}_i(1)}{\bm{a}^1_i} \right\}= \max \left\{ \frac{1}{10}, \frac{1}{1/2},\frac{1/30}{1/2},\frac{1/30}{10} \right\} = 2$, $M^{0,2}_2(1) = \max_{i \in [4]} \left\{ \frac{\alpha^{0,2}_i(1)}{\bm{a}^2_i} \right\} = \max \left\{ \frac{1}{10}, \frac{1}{10},\frac{1/30}{1/3},\frac{1/30}{1/3} \right\} = \frac{1}{10}$, which gives $x_1 + x_2 + \frac{1}{30}x_3 + \frac{1}{30} x_4 - z_1 + \frac{9}{10} z_2 \leq 1$.

\smallskip

\noindent \underline{$j=2$}: Because $r^{0,2}_2 = \frac{1}{10} = r^{0,2}_1$, we have the same inequality $x_1 + x_2 + \frac{1}{30}x_3 + \frac{1}{30} x_4 - z_1 + \frac{9}{10} z_2 \leq 1$.

\smallskip

\noindent \underline{$j=3$}: $r^{0,2}_3 = 30$, $r^{0,2}_3 \bm{a}^2 = 30 \left(10, 10, \frac{1}{3}, \frac{1}{3} \right) = (300, 300, 10, 10)$, $\alpha^{0,2}(3) = (1, 1, 10, 10)$, $M^{0,2}_0(3) = \max_{i \in [4]} \left\{\frac{\alpha^{0,2}_i(3)}{\bm{a}^0_i} \right\} = \max \left\{\frac{1}{1}, \frac{1}{1},\frac{10}{10},\frac{10}{10} \right\} = 1$, $M^{0,2}_1(3) = \max_{i \in [4]} \left\{\frac{\alpha^{0,2}_i(3)}{\bm{a}^1_i} \right\} = \max \left\{\frac{1}{10}, \frac{1}{1/2},\frac{10}{1/2},\frac{10}{10} \right\} = 20$, $M^{0,2}_2(3) = \max_{i \in [4]} \left\{\frac{\alpha^{0,2}_i(3)}{\bm{a}^2_i} \right\} = \max \left\{\frac{1}{10}, \frac{1}{10},\frac{10}{1/3},\frac{10}{1/3} \right\} = 30$, which gives $x_1 + x_2 + 10 x_3 + 10 x_4 - 19 z_1 - 29 z_2 \leq 1$.

\smallskip

\noindent \underline{$j=4$}: Because $r^{0,2}_4 = 30 = r^{0,2}_3$, we have the same inequality $x_1 + x_2 + 10 x_3 + 10 x_4 - 19 z_1 - 29 z_2 \leq 1$.

\bigskip
\noindent \underline{Case $k_1 = 1$, $k_2=2$}:

\noindent \underline{$j=1$}: $r^{1,2}_1 = 1$, $r^{1,2}_1 \bm{a}^2 = 1 \left(10, 10, \frac{1}{3}, \frac{1}{3} \right) = \left(10, 10, \frac{1}{3}, \frac{1}{3} \right)$, $\alpha^{1,2}(1) = \left(10, \frac{1}{2}, \frac{1}{3}, \frac{1}{3} \right)$, $M^{1,2}_0(1) = \max_{i \in [4]} \left\{ \frac{\alpha^{1,2}_i(1)}{\bm{a}^0_i} \right\} = \max \left\{ \frac{10}{1}, \frac{1/2}{1},\frac{1/3}{10},\frac{1/3}{10}\right\} = 10$, $M^{1,2}_1(1) = \max_{i \in [4]} \left\{ \frac{\alpha^{1,2}_i(1)}{\bm{a}^1_i}\right\} = \max \left\{ \frac{10}{10}, \frac{1/2}{1/2},\frac{1/3}{1/2},\frac{1/3}{10}\right\} = 1$, $M^{1,2}_2(1) = \max_{i \in [4]} \left\{ \frac{\alpha^{1,2}_i(1)}{\bm{a}^2_i} \right\} = \max \left\{ \frac{10}{10}, \frac{1/2}{10},\frac{1/3}{1/3},\frac{1/3}{1/3}\right\} = 1$, which gives(after scaling by $\frac{1}{10}$) $x_1 + \frac{1}{20}x_2 + \frac{1}{30}x_3 + \frac{1}{30}x_4 + \frac{9}{10}z_1 + \frac{9}{10} z_2 \leq 1$.

\smallskip

\noindent \underline{$j=2$}: $r^{1,2}_2 = \frac{1}{20}$, $r^{1,2}_2 \bm{a}^2 = \frac{1}{20} \left(10, 10, \frac{1}{3}, \frac{1}{3} \right) = \left(\frac{1}{2}, \frac{1}{2}, \frac{1}{60}, \frac{1}{60} \right)$, $\alpha^{1,2}(2) = \left(\frac{1}{2}, \frac{1}{2}, \frac{1}{60}, \frac{1}{60} \right)$, $M^{1,2}_0(2) = \max_{i \in [4]} \left\{\frac{\alpha^{1,2}_i(2)}{\bm{a}^0_i} \right\} = \max \left\{\frac{1/2}{1}, \frac{1/2}{1},\frac{1/60}{10},\frac{1/60}{10} \right\} = \frac{1}{2}$, $M^{1,2}_1(2) = \max_{i \in [4]} \left\{\frac{\alpha^{1,2}_i(2)}{\bm{a}^1_i} \right\} = \max \left\{\frac{1/2}{10}, \frac{1/2}{1/2},\frac{1/60}{1/2},\frac{1/60}{10} \right\} = 1$, $M^{1,2}_2(2) = \max_{i \in [4]} \left\{\frac{\alpha^{1,2}_i(2)}{\bm{a}^2_i} \right\} = \max \left\{\frac{1/2}{10}, \frac{1/2}{10},\frac{1/60}{1/3},\frac{1/60}{1/3} \right\} = \frac{1}{20}$, which gives(after scaling by $2$) $x_1 + x_2 + \frac{1}{30}x_3 + \frac{1}{30} x_4 - z_1 + \frac{9}{10} z_2 \leq 1$.

\smallskip

\noindent \underline{$j=3$}: $r^{1,2}_3 = \frac{3}{2}$, $r^{1,2}_3 \bm{a}^2 = \frac{3}{2} \left(10, 10, \frac{1}{3}, \frac{1}{3}\right) = \left(15, 15, \frac{1}{2}, \frac{1}{2}\right)$, $\alpha^{1,2}(3) = \left(10, \frac{1}{2}, \frac{1}{2}, \frac{1}{2}\right)$, $M^{1,2}_0(3) = \max_{i \in [4]} \left\{\frac{\alpha^{1,2}_i(3)}{\bm{a}^0_i} \right\} = \max \left\{\frac{10}{1}, \frac{1/2}{1},\frac{1/2}{10},\frac{1/2}{10} \right\} = 10$, $M^{1,2}_1(3) = \max_{i \in [4]} \left\{\frac{\alpha^{1,2}_i(3)}{\bm{a}^1_i} \right\} = \max \left\{\frac{10}{10}, \frac{1/2}{1/2},\frac{1/2}{1/2},\frac{1/2}{10} \right\} = 1$, $M^{1,2}_2(3) = \max_{i \in [4]} \left\{\frac{\alpha^{1,2}_i(3)}{\bm{a}^2_i} \right\} = \max \left\{\frac{10}{10}, \frac{1/2}{10},\frac{1/2}{1/3},\frac{1/2}{1/3} \right\} = \frac{3}{2}$, which gives(after scaling by $\frac{1}{10}$) $x_1 + \frac{1}{20}x_2 + \frac{1}{20} x_3 + \frac{1}{20} x_4 + \frac{9}{10} z_1 + \frac{17}{20} z_2 \leq 1$.

\smallskip

\noindent \underline{$j=4$}: $r^{1,2}_4 = 30$, $r^{1,2}_4 \bm{a}^2 = 30 \left(10, 10, \frac{1}{3}, \frac{1}{3}\right) = (300,300,10,10)$, $\alpha^{1,2}(4) = \left(10, \frac{1}{2}, \frac{1}{2}, 10\right)$, $M^{1,2}_0(4) = \max_{i \in [4]} \left\{\frac{\alpha^{1,2}_i(4)}{\bm{a}^0_i} \right\} = \max \left\{\frac{10}{1}, \frac{1/2}{1},\frac{1/2}{10},\frac{10}{10} \right\} = 10$, $M^{1,2}_1(4) = \max_{i \in [4]} \left\{\frac{\alpha^{1,2}_i(4)}{\bm{a}^1_i} \right\} = \max \left\{\frac{10}{10}, \frac{1/2}{1/2},\frac{1/2}{1/2},\frac{10}{10} \right\} = 1$, $M^{1,2}_2(4) = \max_{i \in [4]} \left\{\frac{\alpha^{1,2}_i(4)}{\bm{a}^2_i} \right\} = \max \left\{\frac{10}{10}, \frac{1/2}{10},\frac{1/2}{1/3},\frac{10}{1/3} \right\}= 30$, which gives(after scaling by $\frac{1}{10}$) $x_1 + \frac{1}{20}x_2 + \frac{1}{20}x_3 + x_4 + \frac{9}{10} z_1 - 2 z_2 \leq 1$. \hfill$\clubsuit$

%%%%%%%%%%%%%%%%%%%%%%%%

\bigskip

\noindent {\bf Acknowledgments.} This work was supported in part by U.S. Office of Naval Research grant N00014-21-1-2135.

\bibliographystyle{alpha}
\bibliography{lowdimdisj}

\end{document}